\documentclass[11pt]{amsart}

\DeclareMathOperator{\Lie}{Lie}

\DeclareMathOperator{\rk}{rk}
\DeclareMathOperator{\ad}{ad}
\DeclareMathOperator{\Der}{Der}
\usepackage{amsmath,amssymb,amsthm,graphics,amscd}
\usepackage{longtable}
\usepackage{cite,xypic}
\usepackage{color}
\usepackage{graphicx}
\usepackage{epsf}
\usepackage{fancyhdr,hyperref}
\usepackage{lscape}
\hypersetup{backref,
colorlinks=true,backref}

  \renewenvironment{thebibliography}[1]{
    \begin{oldthebibliography}{#1}
      \setlength{\parskip}{0ex}
      \setlength{\itemsep}{0ex}
  }
  {
  \end{oldthebibliography}
  }
  
\setlongtables
\begin{document}

\newcounter{rownum}
\setcounter{rownum}{0}
\newcommand{\ab}{\addtocounter{rownum}{1}\arabic{rownum}}

\newcommand{\x}{$\times$}
\newcommand{\bb}{\mathbf}

\newcommand{\Ind}{\mathrm{Ind}}
\newcommand{\Char}{\mathrm{char}}
\newcommand{\hra}{\hookrightarrow}
\newtheorem{lemma}{Lemma}[section]
\newtheorem{theorem}[lemma]{Theorem}
\newtheorem*{TA}{Theorem A}
\newtheorem*{TB}{Theorem B}
\newtheorem*{TC}{Theorem C}
\newtheorem*{CorC}{Corollary C}
\newtheorem*{TD}{Theorem D}
\newtheorem*{TE}{Theorem E}
\newtheorem*{PF}{Proposition E}
\newtheorem*{C3}{Corollary 3}
\newtheorem*{T4}{Theorem 4}
\newtheorem*{C5}{Corollary 5}
\newtheorem*{C6}{Corollary 6}
\newtheorem*{C7}{Corollary 7}
\newtheorem*{C8}{Corollary 8}
\newtheorem*{claim}{Claim}
\newtheorem{cor}[lemma]{Corollary}
\newtheorem{conjecture}[lemma]{Conjecture}
\newtheorem{prop}[lemma]{Proposition}
\newtheorem{question}[lemma]{Question}
\theoremstyle{definition}
\newtheorem{example}[lemma]{Example}
\newtheorem{examples}[lemma]{Examples}
\theoremstyle{remark}
\newtheorem{remark}[lemma]{Remark}
\newtheorem{remarks}[lemma]{Remarks}
\newtheorem{obs}[lemma]{Observation}
\theoremstyle{definition}
\newtheorem{defn}[lemma]{Definition}

  \def\hal{\unskip\nobreak\hfil\penalty50\hskip10pt\hbox{}\nobreak
  \hfill\vrule height 5pt width 6pt depth 1pt\par\vskip 2mm}

\renewcommand{\labelenumi}{(\roman{enumi})}
\newcommand{\Hom}{\mathrm{Hom}}
\newcommand{\Int}{\mathrm{int}}
\newcommand{\Ext}{\mathrm{Ext}}
\newcommand{\opH}{\mathrm{H}}
\newcommand{\D}{\mathcal{D}}
\newcommand{\soc}{\mathrm{Soc}}
\newcommand{\SO}{\mathrm{SO}}
\newcommand{\Sp}{\mathrm{Sp}}
\newcommand{\SL}{\mathrm{SL}}
\newcommand{\GL}{\mathrm{GL}}
\newcommand{\OO}{\mathcal{O}}
\newcommand{\diag}{\mathrm{diag}}
\newcommand{\End}{\mathrm{End}}
\newcommand{\tr}{\mathrm{tr}}
\newcommand{\Stab}{\mathrm{Stab}}
\newcommand{\red}{\mathrm{red}}
\newcommand{\Aut}{\mathrm{Aut}}
\renewcommand{\H}{\mathcal{H}}
\renewcommand{\u}{\mathfrak{u}}
\newcommand{\Ad}{\mathrm{Ad}}
\newcommand{\N}{\mathcal{N}}
\newcommand{\C}{\mathbb{C}}
\newcommand{\Z}{\mathbb{Z}}
\newcommand{\la}{\langle}\newcommand{\ra}{\rangle}
\newcommand{\gl}{\mathfrak{gl}}
\newcommand{\g}{\mathfrak{g}}
\newcommand{\F}{\mathbb{F}}
\newcommand{\m}{\mathfrak{m}}
\renewcommand{\b}{\mathfrak{Borho}}
\newcommand{\p}{\mathfrak{p}}
\newcommand{\q}{\mathfrak{q}}
\renewcommand{\l}{\mathfrak{l}}
\newcommand{\del}{\partial}
\newcommand{\h}{\mathfrak{h}}
\renewcommand{\t}{\mathfrak{t}}
\renewcommand{\k}{\mathfrak{k}}
\newcommand{\Gm}{\mathbb{G}_m}
\renewcommand{\c}{\mathfrak{c}}
\renewcommand{\r}{\mathfrak{r}}
\newcommand{\n}{\mathfrak{n}}
\newcommand{\s}{\mathfrak{s}}
\newcommand{\Q}{\mathbb{Q}}
\newcommand{\z}{\mathfrak{z}}
\newcommand{\pso}{\mathfrak{pso}}
\newcommand{\so}{\mathfrak{so}}
\renewcommand{\sl}{\mathfrak{sl}}
\newcommand{\psl}{\mathfrak{psl}}
\renewcommand{\sp}{\mathfrak{sp}}
\newcommand{\Ga}{\mathbb{G}_a}

\newenvironment{changemargin}[1]{%
  \begin{list}{}{%
    \setlength{\topsep}{0pt}%
    \setlength{\topmargin}{#1}%
    \setlength{\listparindent}{\parindent}%
    \setlength{\itemindent}{\parindent}%
    \setlength{\parsep}{\parskip}%
  }%
  \item[]}{\end{list}}

\parindent=0pt
\addtolength{\parskip}{0.5\baselineskip}

\title{Rigid orbits and sheets in reductive Lie algebras over fields of prime characteristic}

\author{Alexander Premet}
\address{The University of Manchester\\ Oxford Road, M13 9PL, UK} \email{alexander.premet@manchester.ac.uk {\text{\rm(Premet)}}}

\author{David I. Stewart}
\address{King's College\\ Cambridge, UK} \email{dis20@cantab.net {\text{\rm(Stewart)}}}
\pagestyle{plain}
\begin{abstract}Let $G$ be a simple simple-connected algebraic group  over an algebraically closed field $k$ of characteristic $p>0$ with $\g=\Lie(G)$.
We discuss various properties of nilpotent orbits in $\g$, which have previously only been considered over $\C$. Using computational methods, we extend to positive characteristic various calculations of de Graaf with nilpotent orbits in exceptional Lie algebras. 
In particular, we classify those orbits which are reachable,  those which satisfy a certain related condition due to Panyushev, and determine the codimension in the centraliser $\g_e$ of its the derived subalgebra $[\g_e,\g_e]$. Some of these calculations are used 
to show that the list of rigid nilpotent orbits in $\g$,
the classification of sheets of $\g$ and the distribution of the nilpotent orbits amongst them are independent of good characteristic, remaining the same as in the characteristic zero case. We also give a comprehensive account of the theory of sheets in reductive Lie algebras over algebraically closed fields of good characteristic. 
\end{abstract}
\maketitle
\section{Introduction}

Let $G$ be a reductive algebraic group over an algebraically closed field $k$ of characteristic $p\geq 0$ with $\g=\Lie(G)$. We consider various properties of a nilpotent element $e\in\g$ and its centraliser $\g_e$. The element $e$ is called \emph{reachable} if $e\in[\g_e,\g_e]$. It is called \emph{strongly reachable} if $[\g_e,\g_e]=\g_e$, i.e.~the subalgebra $\g_e$ is perfect. If $p$ is a good prime for $G$, it is said to satisfy \emph{Panyushev property} if in the associated grading $\g_e=\bigoplus_{i\geq 0}\g_e(i)$, the Lie subalgebra $\bigoplus_{i\ge 1}\,\g_e(i)$ is generated by $\g_e(1)$. For the case $p=0$, in \cite{DeG13} the author applies various routines in the SLA package of GAP in order to confirm and finish the classification of nilpotent elements which are reachable, strongly reachable or satisfy the Panyushev property. In particular, he confirms that in characteristic $0$, all reachable elements are Panyushev. For $G$ classical and $p=0$ each of three properties above characterises the class of rigid (i.e. non-induced) nilpotent orbits in $\g$. This was established by Yakimova in \cite{Ya}. 

It is the purpose of this article to record extensions to de Graaf's results to deal with the case where $p>0$.
For $G$ exceptional and simply connected (and for any prime $p$) we compute the number $c(e):=\dim\g_e/[\g_e,\g_e]$ for any nilpotent element $e\in\g$. This data is recorded in Tables~\ref{t3} and \ref{t4}. 
It turns out that $c(e)$ is independent of
$p$ provided that $p$ is good for $G$.
Our computations agree with those in \cite{DeG13} made in the characteristic zero case. 
 We also consider another property, which is relevant in the theory of finite $W$-algebras. Let us call a nilpotent element \emph{almost reachable} if it is not reachable, but $\g_e=ke\oplus[\g_e,\g_e]$.

If $G$ is $\SO_n$ or $\Sp_n$ and $p\ne 2$ then \cite[Theorem~3(i)]{PT14} provides a general formula for $c(e)$ in terms of the partition of $n$ attached to $e$. We mention for completeness that for $G=\SL_n$ or $\GL_n$ the value of $c(e)$ is also known for all $e\in\g$ (and all $p$) thanks to the explicit description of $\g_e$ given in \cite{Ya}; see \cite[Remark~1]{PT14} for more detail. So in our paper we mostly deal with the groups of  exceptional types.
The main results of the computational part of our paper are as follows:

\begin{theorem}\label{T1}Let $G$ be a simple, simply connected algebraic group of exceptional type over an algebraically closed field of characteristic $p>0$ and let $e$ be a nilpotent element in $\g=\Lie(G)$. 
Then $e$ is reachable or almost reachable in characteristic $p$ if and only if the orbit of $e$ is listed in the first column of Table \ref{t1} or \ref{t2}. If, moreover, the equality $\g_e=[\g_e,\g_e]$ holds for $e$ in characteristic $p$ then this is indicated in the the third column of Table~\ref{t1} of \ref{t2}. 
\end{theorem}
Tables~\ref{t1} and \ref{t2} show that many new nilpotent orbits become reachable in bad characteristic and on the three occasions this happens when $p$ is good for $G$.
In this case, the new orbits are as follows:
\begin{itemize}
\item[(A)\,]  $p=5$, $G$ is of type ${\rm E_7}$ and  $e$ is of type ${\rm A_3+A_2+A_1}$;
\smallskip
\item[(B)\,] $p=7$, $G$ is of type ${\rm E_7}$ and $e$ is of type ${\rm A_2+3A_1}$;
\smallskip
\item[(C)\,] $p=7$, $G$ is of type ${\rm E_8}$ and $e$ is of type ${\rm A_4+A_2+A_1}$.
\end{itemize}
It is interesting that these 
three orbits are responsible for the existence of new maximal Lie subalgebras in $\g=\Lie(G)$ which have no analogues in characteristic $0$. This will be explained in detail in a forthcoming paper by the authors.

We have also checked the validity of Panyushev property
in good characteristic; see \S\ref{Pan} for more detail on our computations. This property will play an important role 
in proving Humphreys' conjecture on the existence of
$U_\chi(\g)$-modules of dimension $p^{(\dim G\cdot\chi)/2}$ (here $U_\chi(\g)$ stands for the reduced enveloping algebra of $\g$ associated with a linear function $\chi\in\g^*$). We denote by $G_\C$ the complex simple algebraic group of the same type as $G$ and let
$\g_\C=\Lie(G_\C)$. Recall that if $p$ is good for $G$ then the nilpotent orbits in $\g$ have the same labels as those in $\g_\C$; see \cite[pp.~401--407]{Car93}, for example.

\begin{theorem}\label{T2} Let $G$ be as in Theorem~\ref{T1} and suppose $p$ is a good prime for $G$. Then  an element $e$ in a nilpotent orbit $\OO$ satisfies the Panyushev condition if and only if the nilpotent orbit in 
$\g_\C$ which has the same label as $\OO$  consists of reachable elements. 
In other words, $e$ satisfies the Panyushev condition if and only if it is reachable and not listed in cases~(A), (B) or (C) above.\end{theorem}

It follows from Theorems~\ref{T1} and \ref{T2} that the Panyushev and strong reachability conditions are independent of good characteristic.

A nilpotent element $e\in\g$ is called {\it rigid} if it cannot be obtained by Lusztig--Spaltenstein induction from a proper Levi subalgebra of $\g$. 
Arguing as in \cite[3.2]{P10} one can reduce proving Humphreys' conjecture to the case where $\chi\in\g^*$ corresponds to a rigid nilpotent element of $\g$ under a $G$-equivariant isomorphism $\g\cong\g^*$. 
In the characteristic zero case, all rigid nilpotent orbits in exceptional Lie algebras are classified by Elashvili
\cite{Elashvili} and his computations are recently double-checked in \cite{dG-E} by using GAP. 
Since the 
way this is done in {\it loc.\,cit.} relies heavily on the characteristic $0$ hypothesis, we present in Section~\ref{Sec2} a classification of rigid nilpotent orbits
in exceptional Lie algebras valid over any algebraically closed field of good characteristic (at some point we have to rely on GAP as well).  

\begin{theorem}\label{T3}
Let $G$ be as in Theorem~\ref{T1} and suppose $p$ is a good prime for $G$. Then a nilpotent orbit $\OO$ is rigid in $\g$ if and only if so is the nilpotent orbit in $\g_\C$ which has the same label as $\OO$.
\end{theorem}

Given $m\in\mathbb{N}$ we let $\g_{(m)}$ denote the set of all $x\in \g$ such that $\dim\g_x=m$.
A subset $\mathcal{S}\subset\g$ is called a {\it sheet} of $\g$ if it coincides with
an irreducible component of one of the quasi-affine varieties $\g_{(m)}$.
According to a classical result of Borho \cite{Borho} the sheets of $\g_\C$ are parametrised by the $G_\C$-conjugacy classes of pairs $(\l_\C,\OO_0)$ where $\l_\C$ is a Levi subagebra of $\g_\C$ and $\OO_0$ is a rigid nilpotent orbit in $\l$. In Section~\ref{Sec2}, we give a comprehensive account of the theory of sheets in reductive Lie algebras $\g=\Lie(G)$ which satisfy the standard hypotheses 
(when $G$ is simple and not of type ${\rm A}_{rp-1}$
this is equivalent to saying that $p$ is a good prime for $G$). In particular, we show that every sheet of $\g$ contains a unique nilpotent orbit and Borho's classification of sheets remains valid under our assumptions on $G$.

For any sheet $\mathcal{S}$ of a complex exceptional Lie algebra $\g_\C=\Lie(G_\C)$, de Graaf and Elashvili determine the weighted Dynkin diagram of the unique nilpotent orbit $\N(\g_\C)\cap \mathcal{S}$ and give a very nice representative $e_{\Gamma}=\sum_{\gamma\in\Gamma}\,e_\gamma$ in $\N(\g_\C)\cap \mathcal{S}$ compatible with the combinatorial data that defines $\mathcal{S}$. 
Here $\Gamma=\Gamma(\mathcal{S})$ is a subset of roots the root system $\Phi$ of $\g_\C$ and $e_\gamma$ is a root vector of $\g_\C$ corresponding to $\gamma\in\Phi$. 
Each set $\Gamma$ is linearly independent in the vector space $\Q\Phi$ and the $G_\C$-orbit of $e_\Gamma$ is independent of the choices of root vectors 
$e_\gamma\in(\g_\C)_\gamma$.
There is a natural way to attach to $\Gamma$ a graph  $D(\Gamma)$, and it turned out (not surprisingly) that
in many cases the graphs thus obtained are admissible in the sense of Carter; cf. \cite{Car72} and the last column of the tables in \cite{dG-E}.

Since we have a natural analogue of $e_\Gamma$ in $\g=\Lie(G)$, where $G$ is a simple algebraic $k$-group of the same type as $G_\C$, and the discussion above indicates that there  is a  natural bijection between the sheets of $\g$ and $\g_\C$, one wonders whether a $k$-analogue of $e_\Gamma$ still belongs to the sheet of $\g$ given by the same data as $\mathcal{S}$.
We verify that for exceptional groups this is indeed the case and from the validity of the representatives of \cite{dG-E} in good characteristic we then deduce the following:
\begin{theorem}\label{thm:distrib} If $G$ is an exceptional algebraic $k$-group and $p={\rm char}(k)$ is a good prime for $G$, then the distribution of nilpotent orbits amongst the sheets of $\g$ agrees with that of $\g_\C$ and can be read off from the tables in \cite{dG-E}.\end{theorem}
For many elements $e_\Gamma$ our proof of Theorem~\ref{thm:distrib} is computer-free. For example, this is the case when $D(\Gamma)$ is a disjoint union of Dynkin graphs. However, at the end of the day we do rely on GAP (as did de Graaf and Elashvili) and we have decided to run our programme on all elements $e_\Gamma$ in order to obtain an independent confirmation of the computations in \cite{dG-E}.
At some point in the proof we have to show that the adjoint orbits of $e_\Gamma\in\g$ and its counterpart $e_{\Gamma,\C}\in\g_\C$ have the same Dynkin labels. We
deduce this by asking GAP to determine the Jordan block structure of each element $\ad\,e_\Gamma$ and a chosen representative of each nilpotent orbit in $\g$. 
Since in turned out, for $G$ exceptional, that the Jordan block decomposition of $\ad\,e$ identifies the orbit of $e$ almost uniquely (except when $p=7$ where the Jordan block decompositions of nilpotent elements of type $B_3$ and $C_3$ in Lie algebras of type ${\rm F}_4$ are the same),
this enabled us to determine the Dynkin label of $e_\Gamma$ and finish the proof of Theorem~\ref{thm:distrib}.

{\bf Acknowledgement.} We would like to thank Micha{\"e}l Bulois, Simon Goodwin, Willem de Graaf, George Lusztig and Dan Nakano for very useful discussions and email correspondence on the subject
of this paper.

\section{Sheets and induced nilpotent orbits in good characteristic}\label{Sec2}
\subsection{The standard hypotheses}\label{SH}
Let $G$ be a connected reductive group over an algebraically closed field $k$ of characteristic $p>0$
and $\g=\Lie(G)$.
Being the Lie algebra of an algebraic $k$-group, $\g$ carries a canonical $p$-th power map $x\mapsto x^{[p]}$ equivariant under the adjoint action of $G$. Given $x\in \g$ we denote by
$\g_x$ (resp., $G_x$), the centraliser of $x$ in $\g$ (resp., $G$).
 In order to apply \cite[Theorem~A]{P03} to all Levi subalgebras of $\g$ we shall assume, unless otherwise specified, that $p$ is a good prime for $G$, the derived subgroup of $G$ is simply connected and $\g$ admits an 
$(\Ad\, G)$-invariant non-degenerate symmetric bilinear form. 
We fix such a form  and call it $\kappa$.

The above conditions on $G$ are often referred as the {\it standard hypotheses}. If they hold then  for any $x\in G$ we have the equality $\g_x=\Lie(G_x)$ (the latter is sometimes expressed by saying that the 
scheme-theoretic centraliser of $e$ in $G$ is smooth).
One also knows that if $G$ satisfies the standard hypotheses then the centraliser of any semisimple element of the restricted Lie algebra $\g$
is a Levi subalgebra of $\g$. It is worth remarking that any simple, simply connected algebraic $k$-group of type other than ${\rm A}_{rp-1}$, where $p={\rm char}(k)$, satisfies the standard hypotheses if and only if $p$ is a good prime for $G$.

If $G$ satisfies the standard hypotheses then so does any Levi subgroup $L$ of $G$. It is well known that up to conjugacy any such subgroup is associated with a subset of a chosen basis of simple roots of $G$. Let $\Phi$ be root system of $G$ with respect to a maximal torus $T$ of $G$ and let $\Pi$ be a basis of simple roots in $\Phi$.  
Let $\t=\Lie(T)$. If $L$ is the standard Levi subgroup of $G$ associated with a subset $\Pi_0$ of $\Pi$ 
and $\l=\Lie(L)$ then $\t_0:= [\l,\l]\cap \t$ has dimension equal to $\dim \t-
|\Pi_0|$ (this follows from the fact that the derived subgroup of $L$ is simply connected). On the other hand, the orthogonal complement $\t_0^\perp:=\{t\in\t\,|\,\,\kappa(t,\t_0)=0\}$ contains $\z(\l)$ which forces $\dim\z(\l)\le 
\dim\t_0^\perp=\dim \t-|\Pi_0|$. Since  $\Lie(Z(L))\subseteq\z(\l)$ has dimension equal to $\dim \t-|\Pi_0|$ this shows that $\z(\l)=\Lie(Z(L))$ for any Levi subgroup $L$ of $G$. Furthermore, if $\Pi_0$ has no components of type ${\rm A}_{rp-1}$ then the structure theory of reductive Lie algebras yields $\z([\l,\l])=0$ implying $\l=[\l,\l]\oplus\z(\l)$.

Let $G_1,\ldots, G_s$ be the simple components of 
$G$ and $\g_i=\Lie(G_i)$. 
Let $T_i=T\cap G_i$ and $\t_i=\Lie(T_i)$.
If $G_i$ is not of type ${\rm A}_{rp-1}$ then it is well known that $\g_i$ is a simple Lie algebra and the restriction of $\kappa$ to $\g_i$ is non-degenerate. If $G_l$ has type ${\rm A}_{m(l)}
$ with $p\vert (m(l)+1)$ then $\g_l\cong\sl_{m(l)+1}$ 
as restricted Lie algebras and 
a nonzero scalar multiple, $\kappa'$, of the restriction of $\kappa$ to
$\g_l$ coincides with the trace form of the standard (vector) representation of $G_l$. Let $\alpha_1,\ldots,\alpha_{m(l)}$ be a basis of simple roots of the root system of type ${\rm A}_{m(l)}$ in Bourbaki's numbering. Then $\t_l$ has basis
$h_1,\ldots,h_{m(l)}$ such that
$\alpha_i(h_i)=2=\kappa'(h_i,h_i)$ for all $i$ and we also have that
$\kappa'(h_i,h_{i+1})=-1$ for $1\le i\le m(l)-1$, $\kappa'(h_i,h_j)=0$ for $j\not\in\{i-1,j,i+1\}$, and $h_i=[e_{\alpha_i},e_{-\alpha_i}]$ for some root vectors $e_{\pm\alpha_i}\in \g_l$ with $\kappa'(e_{\alpha_i},e_{-\alpha_i})=1$.
Suppose $\kappa_{\vert\g_l}=b\cdot\kappa'$ where $b=b(l)\in k^\times$. Since the restriction of $\kappa$ to $\t$ is 
non-degenerate, there exists an element $h_0\in \t$ orthogonal to all $\g_i$ with $i\ne m(l)$ and such that 
$\kappa(h_0,h_1)=b$ and $\kappa(h_0,h_i)=0$ for $2\le i\le m(l)$. 
In general, $h_0$ is not unique and we only know that $\kappa(h_0,h_0)=a$ for some $a=a(l)\in k$. 

Let $C_n$ denote the Cartan matrix of the root 
system of type ${\rm A}_n$ with entries reduced modulo $p$, so that $\det(C_n)=n+1\ ({\rm mod}\, p)$. 
Then the restriction of $\kappa$ to the $k$-span of $h_0,h_1,\ldots, h_{m(l)}$ is represented by the  matrix
\[A= \left[ \begin{array}{rrrrrr}
a & b & 0&\cdots &0&0 \\
b &2b &-b &\cdots &0&0 \\
0 & -b & 2b&\cdots&0 &0\\
\vdots&\vdots&\vdots&\cdots&\vdots&\vdots\\
0&0&0&\cdots&2b&-b\\
0&0&0&\cdots&-b&2b
\end{array} \right]\]
of order $m(l)+1$ which is non-singular because 
$\det(A)=ab^{m(l)}\det (C_{m(l)})-b^{m(l)+1}\det(C_{m(l)-1})=b^{m(l)+1}\ne 0$
by the first row Laplace expansion.
In particular, this shows that $h_0,h_1,\ldots, h_{m(l)}$ are linearly independent.
The $\g$-invariance of $\kappa$ yields
\begin{align*}
b&=\kappa(h_0,h_1)=
\kappa([h_0,e_{\alpha_1}],e_{-\alpha_1})=
({\rm d}\alpha_1)_e(h_0)\kappa(e_{\alpha_1},e_{-\alpha_1})=
b\cdot ({\rm d}\alpha_1)_e(h_0),\\
0&=\kappa(h_0,h_i)=
\kappa([h_0,e_{\alpha_i}],e_{-\alpha_i})=
({\rm d}\alpha_i)_e(h_0)\kappa(e_{\alpha_1},e_{-\alpha_1}),\qquad \, 2\le i\le m(l).
\end{align*}
It follows that $[h_0,e_{\pm\alpha_1}]=\pm e_{\alpha_1}$ and 
$[h_0,e_{\pm\alpha_i}]=0$ for $2\le i\le m(l)$.

We now set $\tilde{\g}_l:=kh_0\oplus \g_l$ if $G_l$ is of type ${\rm A}_{rp-1}$ and $\tilde{\g}_l:=\g_l$ otherwise. 
It is immediate from  above discussion  that if 
$G_l$ is of type ${\rm A}_{rp-1}$ then $\tilde{\g}_l$ is an $(\Ad\,G)$-stable ideal of $\g$
isomorphic to $\gl_{rp}$ as an abstract Lie algebra and the restriction of $\kappa$ to $\tilde{\g}_l$ is
non-degenerate. Moreover, the adjoint action of $G_l$ on  $\tilde{\g}_l$ is induced by that of $GL_{rp}$. 
Since $\kappa_{\vert\tilde{\g}_l}$  is
non-degenerate for all $l$ we can attach $h_0$'s to different components of type ${\rm A}_{rp-1}$ in such a way that they form an orthogonal set. Note also that the condition $\kappa(h_0,\g_i)=0$ forces $[h_0,\g_i]=0$
due to the $\g$-invariance of $\kappa$.
We thus deduce that $\g=\tilde{\g}_1\oplus\cdots\oplus\tilde{\g}_s\oplus\z$
where $\z$ is a central subalgebra of $\g$ contained in $\t$ and all direct summands in this decomposition are $(\Ad\, G)$-stable (in \cite[2.9]{Jan04} this result is mentioned without a proof).

Of course, the above has bearing on the description of the sheets of $\g$ which turns out to be the same as in the characteristic $0$ case. Since we shall require some rather detailed results on sheets of $\g$ and they are hard to find in the literature under our assumption on $G$, some short proofs will be given below.

\subsection{Sheets and decomposition classes}
Given $m\in\Z_{\ge 0}$ we denote by $\g_{(m)}$
the set of all $x\in\g$ for which $\dim G_x=m$. Each subset $\g_{(m)}$ is locally closed in the Zariski topology of $\g$ and the irreducible components of the $\g_{(m)}$'s are called the {\it sheets} of $\g$. 
\begin{remark}\label{r1} In prime characteristic there are two meaningful ways to define sheets of $\g$ and $\g^*$. Given $m\in\Z_{\ge 0}$ we let $\g_{[m]}$ (resp. $\g_{[m]}^*$) denote the subset of $\g$ (resp. $\g^*$) consisting of all elements whose stabiliser in $\g$
has dimension $m$. We call the irreducible components of the quasi-affine varieties $\g_{[m]}$ (resp. $\g_{[m]}^*$) the {\it infinitesimal sheets} of $\g$ (resp. $\g^*$). 
All infinitesimal sheets are $G$-stable, quasi-affine varieties.
If $G$ satisfies the standard hypotheses all these notions coincide, but in general there will be some subtle differences well worth investigating. 
We call a $G$-orbit $\g$ (resp. $\g^*$) {\it infinitesimally isolated} if it forms a single sheet in $\g$ (resp. $\g^*$). The problem of classifying all infinitesimally isolated nilpotent orbits in reductive Lie algebras is  interesting and wide open at the moment.
As an example, if $k$ has characteristic $p$ and $\g=\mathfrak{pgl}_p$ then the regular nilpotent orbit $\OO_{\rm reg}$ in $\g$ coincides with $\g_{[p]}$ and therefore forms a single infinitesimal sheet of $\g$.  This can be deduced from the following elementary result of linear algebra: if $X,Y\in \mathfrak{gl}_p$ and
$[X,Y]=I_p$ then both $X$ and $Y$ have a single Jordan block
of size $p$ and hence are mapped by the canonical homomorphism $\mathfrak{gl}_p\to \g$ to a commuting pair of elements in $\OO_{\rm reg}$. So in contrast with the classical case the regular nilpotent orbit in $\mathfrak{pgl}_p$ is infinitesimally isolated. 
On the other hand, if $\g=\sl_p$ then $\g_{[p]}$  consists of all $p\times p$ matrices $X$ such that $k^p$ is an indecomposable $k[X]$-module. From this it follows that $\g_{[p]}$ is a single infinitesimal sheet of $\g$, but
$\OO_{\rm reg}\subset \g_{[p]}$ is not infinitesimally isolated in $\g$. In fact, there are no infinitesimally isolated orbits in $\g=\sl_p$ at all, because $I_p\in\g$ and all sheets of $\g$ are invariant under the affine translations $x\mapsto x+\alpha I_p$ where $\alpha\in k$.
Finally, we mention that the study of infinitesimal sheets in $\g^*$ finds interesting applications in the modular representation theory of reductive Lie algebras; see \cite[Remark~5.6]{PSk}.
\end{remark}
 
Remark~\ref{r1} indicates that one cannot expect a uniform behaviour of sheets for all reductive Lie algebras over algebraically closed fields (even in good characteristic) and  partially justifies the assumptions we imposed on $\g$.
 
If $\l$ is a Levi subalgebra of $\g$ then its centre $\z(\l)$ is a toral subalgebra of $\g$ and $\l=\c_\g(\z(\l))$. We denote by $\z(\l)_{\rm reg}$ the set of all $t\in\z(\l)$ such that $\l=\c_\g(t)$. This is a nonempty Zariski open subset of $\z(\l)$. Given a nilpotent element $e_0\in \l$ we set $$\mathcal{D}(\l,e_0)\,:=\,(\Ad\, G)\cdot\big(e_0+\z(\l)_{\rm reg}\big)$$ and we call $\mathcal{D}(\l,e)$ the {\it decomposition class} of $\g$ associated with the pair $(\l,e_0)$.
If $x=x_s+x_n$ is the Jordan decomposition of $x\in \mathcal{D}(\l,e_0)$ then $x_s$ (resp., $x_n$) is $G$-conjugate to an element of $\z(\l)_{\rm reg}$ (resp., $e_0$). It follows that $\dim \g_x=\dim \l_{e_0}$ for every such $x$. Since there are finitely many nilpotent $L$-orbits in $\l=\Lie(L)$ and the number of 
$G$-conjugacy classes of Levi subalgebras of $\g$ is finite,
$\g$ contains finitely many decomposition classes. 
Each of them is an irreducible, locally closed subset of $\g$ contained in one of the $\g_{(r)}$'s and hence lies in a sheet of $\g$.  As the sheets are irreducible and their union is the whole of $\g$, there can be only finitely many, and so every sheet contains a unique open decomposition class of $\g$.
\begin{remark}
Since the number of nilpotent orbits is finite in all characteristics, the above discussion is essentially valid for all reductive Lie algebras over algebraically closed fields. However, in bad characteristic the notion of a decomposition class in $\g$ will involve the centraliser in $G$ of a semisimple element of $\g$ which is not always a Levi subgroup of $G$. In general, the connected component of that group is $G$-conjugate to a subgroup of $G$ obtained by base change from some standard {\it pseudo}-Levi subgroup of $G_\C$. 
\end{remark}
\subsection{The Zariski closure of a decomposition class}
Borho's classification of sheets of $\g$ is based on an important result of Borho--Kraft which provides a description of the Zariski closure of a decomposition class in $\g$; see \cite[Theorem~5.4]{BK}. Since
a modular version of this result valid under our assumptions on $G$ seems to be missing in the literature, a short proof will be given below.  

Let $\mathcal{D}(\l,e_0)$ be a decomposition class of $\g$ and let $\lambda\in X_*(L)$ be an optimal cocharacter for $e_0\in\l$ as defined in \cite[2.3]{P03}.
We adopt the notation of {\it loc.\,cit.} and write $\l(\lambda,i)$ for the $i$-weight space of the 
$1$-dimensional torus $\lambda(k^\times)$ acting on $\l$. By construction, $e_0\in\l(\lambda,2)$.
Given $k\in\Z$ we set 
$\l(\lambda,\ge k):=\bigoplus_{i\ge k}\,\l(\lambda,i)$.
Then $\p_\l(\lambda):=\l(\lambda,\ge 0)$ is the optimal parabolic subalgebra of $e_0\in \l$. We write $P_L(\lambda)$ for the corresponding parabolic subgroup of $L$. Note that $P_L(\lambda)=Z_L(\lambda)R_u(P_L)$ where $\Lie(Z_L(\lambda))=\l(\lambda,0)$ and $\Lie(R_u(P_L(\lambda)))=\l(\lambda,\ge 1)$. As explained in \cite[p.~346]{P03}, there exists a nonzero regular function $\varphi$ on $\l(\lambda,2)$ semi-invariant under the adjoint action of $Z_L(\lambda)$ and such that the orbit $(\Ad\,Z_L(\lambda))\cdot e_0$ coincides with the principal open subset $\l(\lambda,2)_\varphi=\{x\in \l(\lambda,2)\,|\,\,\varphi(x)\ne 0\}$ of $\l(\lambda,2)$.

Let $\g=\n_-\oplus\l\oplus\n_+$ be a parabolic decomposition of $\g$ associated with the Levi subalgebra $\l$. Then there exists a parabolic subgroup 
$P=LN_+$ such that $N_+=R_u(P)$ and $\n_+=\Lie(N_+)$.
We set $P(\l,e_0):=P_L(\lambda)N_+$, a parabolic subgroup of $G$ contained in $P$. It is easy to see that $\Lie(P(\l,e_0))=\p_\l(\lambda)\oplus\n_+$ and
$$\r(\l,e_0)\,:=\,\z(\l)\oplus\l(\lambda,\ge 2)\oplus\n_+$$ is a solvable ideal of $\Lie(P(\l,e_0))$ invariant under the adjoint action of $P(\lambda,e_0)$. Identifying direct sums with direct products we put
$$\r(\l,e_0)_{\rm reg}\,:=\,\z(\l)_{\rm reg}\oplus\l(\lambda,2)_\varphi\oplus \l(\lambda,\ge 3)\oplus\n_+.$$ By construction, $\r(\l,e_0)_{\rm reg}$ is a Zariski open subset of $\r(\l,e_0)$ invariant under the adjoint action of $P(\l,e_0)$. We mention for further references that $\z(\l)=\Lie(Z_G(L))$ and there exists 
a connected algebraic subgroup $R(\l,e_0)$ of $G$ with a maximal torus $Z_G(L)^\circ$ such that
$\r(\l,e_0)=\Lie(R(\l,e_0))$.
\begin{theorem}\label{thm2.3}
In the above notation, $\mathcal{D}(\l,e_0)=(\Ad\,G)\cdot\r(\l,e_0)_{\rm reg}$ and the Zariski closure of
$\mathcal{D}(\l,e_0)$ in $\g$ coincides with $(\Ad\,G)\cdot \r(\l,e_0)$.
\end{theorem}
\begin{proof}
Let $z\in\z(\l)_{\rm reg}$. Then $(\Ad\,P_L(\lambda))\cdot(z+e_0)=
z+(\Ad\,P_L(\lambda))\cdot e_0$.
Since
$[\p_\l(\lambda),e_0]=\l(\lambda,\ge2)$ by \cite[Theorem~2.3(iv)]{P03} and the orbit $(\Ad\,R_u(P_L(\lambda)))\cdot e_0\subseteq e_0+\l(\lambda,\ge 3)$ is Zariski closed by Rosenlicht's theorem \cite[Theorem~2]{Ro61},  the equality
$(\Ad\,R_u(P_L(\lambda)))\cdot e_0\,=\,e_0+\l(\lambda,\ge 3)$ must hold. It implies that $$
(\Ad\,P_L(\lambda))\cdot(z+e_0)=z+(\Ad\,Z_L(\lambda))
\cdot (e_0+\l(\lambda,\ge 3))=z+\l(\lambda,2)_\varphi+\l(\lambda,\ge 3).$$
As $z\in\z(\l)_{\rm reg}$ and $e_0$ is a nilpotent element of $\l$, it must be the endomorphism $\ad(z+e_0)$ acts invertibly on $\n_+$. As $(\Ad\,N_+)\cdot (z+e_0)\subseteq z+e_0+\n_+$, another application of Rosenlicht's theorem yields $(\Ad\,N_+)\cdot (z+e_0)=z+e_0+\n_+$. It follows that
\begin{eqnarray}\label{e0}(\Ad\,P(\l,e_0))\cdot (z+e_0)&=&(\Ad\,P_L(\lambda))\cdot\big((\Ad\,N_+)\cdot (z+e_0)\big)\\
\nonumber &=&(\Ad\,P_L(\lambda))\cdot(z+e_0+\n_+)
\,=\,\n_+ +(\Ad\,P_L(\lambda))\cdot(z+e_0)\\
\nonumber&=&
z+\l(\lambda,2)_\varphi+\l(\lambda,\ge 3)+\n_+.
\end{eqnarray}
As this holds for all $z\in\z(\l)_{\rm reg}$ we now deduce that $(\Ad\,P(\l,e_0))\cdot (\z(\l)_{\rm reg}+e_0)=\r(\l,e_0)_{\rm reg}$ and
$\mathcal{D}(\l,e_0)=
(\Ad\,G)\cdot\r(\l,e_0)_{\rm reg}$.
As $\r(\l,e_0)_{\rm reg}$ is Zariski open in $\r(\l,e_0)$, the Zariski closure of $\mathcal{D}(\l,e_0)$ contains $(\Ad\,G)\cdot \r(\l,e_0)$. Since
the latter set is the image of the action morphism
$$G\times^{P(\l,e_0)}\r(\l,e_0)\,\longrightarrow\,\,\g,\qquad(g,r)\mapsto (\Ad\,g)\cdot r,$$ and $P(\l,e_0)$ is a parabolic subgroup of $G$, the set $(\Ad\,G)\cdot \r(\l,e_0)$ is Zariski closed in $\g$; see \cite[Lemma~8.7(c)]{Jan04}, for example.
Since $(\Ad\,G)\cdot\r(\l,e_0)_{\rm reg}$ is  Zariski dense in $(\Ad\,G)\cdot \r(\l,e_0)$ we  conclude that 
$\overline{\mathcal{D}(\l,e_0)}=
(\Ad\,G)\cdot \r(\l,e_0)$ thereby
 completing the proof.
\end{proof}
\begin{remark}\label{r4}
Since $\lambda\in X_*(T)$ is an optimal cocharacter for $e_0\in\N(\l)$ in the sense of the Kempf--Rousseau theory, it follows from the proof of Theorem~\ref{thm2.3} that for any $x\in\r(\l,e_0)_{\rm reg}$ the inclusion
$G_x\subset P(\l,e_0)$ holds. In the case where ${\rm char}(k)=0$ this was first observed in \cite[Zuzatz~5.5(e)]{BK}.
\end{remark}
\subsection{Restricting the adjoint quotient map to a decomposition class} Let $\t=\Lie(T)$ be a maximal toral subalgebra of $\g$ containing $\z(\l)$ and let $W=N_G(T)/T$ be the Weyl group of $G$. Let $n=\dim G$ and $\ell=\rk G$. 
Under our assumptions on $G$ the invariant ring
$k[\g]^G$ is freely generated by $\ell$ homogeneous regular functions $f_1,\ldots, f_\ell$ and the restriction map $k[\g]\twoheadrightarrow k[\t]$
induces a natural isomorphism of invariant rings $j\colon\,k[\g]^G\to\,k[\t]^W$.
Furthermore, the categorical quotient morphism
$F\colon\,\g\to\g/\!\!/G\cong\t/W$ sending any $x\in\g$ to
$\big(f_1(x),\cdots,f_l(x)\big)\in{\Bbb A}^\ell$ is flat, surjective, and all its fibres are irreducible complete intersections of dimension $n-\ell$ in $\g$ consisting of finitely many $G$-orbits. 
If $x\in \g$ then $f(x)=f(x_s)$ for every $f\in k[\g]^G$,
and for every $\xi\in{\Bbb A}^\ell$  there is a
unique orbit $W\cdot t\subset \t$ such that
the fibre $F^{-1}(\xi)$ consists of all $x\in \g$ such that 
$(\Ad\,g)\cdot x_s=t$ for some $g\in G$. The special fibre $F^{-1}(0)$ coincides with the nilpotent cone $\N(\g)$ of $\g$.
All these results are well known and can be found in \cite[Sect.~7]{Jan04}, for example. 

In this subsection we are concerned with the restriction of the adjoint quotient morphism $F$ to the Zariski closure of a decomposition class.
\begin{prop}\label{prop2.4} Let $\bar{F}$ denote the restriction of the adjoint quotient morphism $F\colon\g\to\t/W$ to $\overline{\mathcal{D}(\l,e_0)}$. Then
all irreducible components of the fibres of $\bar{F}$
 have dimension $n-\dim\l_{e_0}$, the special fibre
$\bar{F}^{-1}(0)=\overline{\mathcal{D}(\l,e_0)}\cap\N(\g)$ is irreducible, and $\dim \mathcal{D}(\l,e_0)=n-\dim\l_{e_0}+\dim\z(\l)$.
\end{prop}
\begin{proof}
To ease notation we  put $n(\l,e_0)=\dim\mathcal{D}(\l,e_0)$. Let $X$ be a fibre of $\bar{F}$ and $x\in X$.
By Theorem~\ref{thm2.3}, we may assume without loss that $x\in \r(\l,e_0)$. Since $\r(\l,e_0)=\Lie(R(\l,e_0))$ and $Z_G(L)^\circ$ is a maximal torus of $R(\l,e_0)$, the standard fact in the theory of algebraic groups that maximal tori are conjugate yields that the semisimple part of $x$ is $R(\l,e_0)$-conjugate to an element of $\z(\l)=\Lie(Z_G(L)^\circ)$; see \cite[11.8]{Borel}. 
In view of the discussion above this implies that
\begin{equation}\label{e1}
X\,=\bigcup_{t\in\, (W\cdot x_s)\,\cap\, \z(\l)}
(\Ad\,G)\cdot \big(t+\l(\lambda,\ge 2)+\n_+\cap \g_t\big).\end{equation} (As $t\in\z(\l)$ we have $\l(\lambda,\ge 2)\subset \g_t$.) As $\z(\l)\subset\overline{\mathcal{D}(\l,e_0)}$
this implies that $\bar{F}$ maps $\overline{\mathcal{D}(\l,e_0)}$ {\it onto} $\z(\l)/W$, the image of $\z(\l)$ in $\t/W$. Since $W$ is a finite group the quotient morphism $\t\to\t/W$ is finite, hence closed
and has finite fibres. So $\z(\l)/W$ is a closed subset of $\t/W$ of dimension $d:=\dim\z(\l)$. 
It follows that the minimal dimension of the fibres of
$\bar{F}$ equals $n(\l,e_0)-d$. 

On the other hand, it follows from \cite[Theorem~14.8(a)]{Eisenbud} that
the union of those irreducible components of the fibres of $\bar{F}$ that have dimension $> n(\l,e_0)-d$
is Zariski closed in $\overline{\mathcal{D}(\l,e_0)}$. 
That union cannot coincide with $\overline{\mathcal{D}(\l,e_0)}$ by the preceding remark.
Since all irreducible components of the fibres of $\bar{F}$
contain open $G$-orbits
and $\mathcal{D}(\l,e_0)$ is dense in
$\overline{\mathcal{D}(\l,e_0)}$, there exists $z\in \z(\l)_{\rm reg}$ such that $n-\dim \g_{z+e_0}=n(\l,e_0)-d$. Since $z$ is the semisimple part of $z+e_0$ we have the equality $\g_{z+e_0}=\l_{e_0}$. So $\dim \g_{z+e_0}=\dim\l_{e_0}$ implying $n(\l,e_0)=n-\dim\l_{e_0}+d$. This proves the last statement of the proposition.

If one of the fibres of $\bar{F}$ has an irreducible component of dimension $>n(\l,e_0)-d=n-\dim\l_{e_0}$, then by the same token  it contains an open $G$-orbit of the same dimension and hence an element $v\in\overline{\mathcal{D}(\l,e_0)}$ such that $\dim \g_v<\dim\l_{e_0}$.
Since the set $\{x\in \overline{\mathcal{D}(\l,e_0)}\,|\,\,\dim \g_x\ge \dim\l_{e_0}\}$ is Zariski closed by Chevalley's semi-continuity theorem and  
$\dim\g_x=\dim\l_{e_0}$ for all $x\in{\mathcal{D}(\l,e_0)}$  it follows that $\dim \g_x\ge\dim\l_{e_0}$ for all $x\in \overline{\mathcal{D}(\l,e_0)}$. This contradiction shows that all irreducible components of the fibres of $\bar{F}$ have the same dimension.

If $X$ is the special fibre of $\bar{F}$, then $x_s=0$ and (\ref{e1}) gives
\begin{equation}\label{e2}
\overline{\mathcal{D}(\l,e_0)}\cap\N(\g)=(\Ad\,G)
\cdot (\l(\lambda,\ge 2)+\n_+).
\end{equation} Therefore, the variety $\overline{\mathcal{D}(\l,e_0)}\cap\N(\g)$ is irreducible which completes the proof.
\end{proof} 
 \begin{remark}\label{r5} 
If $\mathcal{S}$ is a sheet of $\g$ whose open decomposition class equals $\mathcal{D}(\l,e_0)$ then by maximality $\mathcal{S}$ must coincide with the set of {\it all} $G$-orbits of maximal dimension contained in $\overline{\mathcal{D}(\l,e_0)}$.
Therefore, Proposition~\ref{prop2.4} implies that every sheet of $\g$ contains a unique nilpotent orbit. Since $W$ is a finite group, the quotient morphism $\t\to\t/W$ is open and closed. As $\t_{\rm reg}$ is open in $\t$, the set $\t_{\rm reg}/W$ is open in $\t/W$. So the proof of
Proposition~\ref{prop2.4} also shows that
any decomposition class $\mathcal{D}(\l,e_0)=\bar{F}^{-1}(\t_{\rm reg}/W)$ of $\g$ is Zariski open in its closure. 
It is worth mentioning that all $G$-stable pieces in the union~(\ref{e1}) are Zariski closed in $\g$. This follows from \cite[Lemma~8.7(c)]{Jan04} (see the end 
of the proof of Theorem~\ref{thm2.3} for a similar argument).
If $\mathcal{D}(\l,e_0)$ is the open decomposition class
of a sheet $\mathcal{S}$ and $\OO$ is the nilpotent orbit contained in $\mathcal{S}$ then it follows from Proposition~\ref{prop2.4} that $\dim\z(\l)=\dim\mathcal{S}-
\dim\OO.$ This number is called the {\it rank} of $\mathcal{S}$ and denoted $\rk(\mathcal{S})$. It is immediate from Proposition~\ref{prop2.4} that under our ssumptions on $G$ a sheet $\mathcal{S}$ of $\g$ is a single nilpotent $G$-orbit if and only if $\rk(\mathcal{S})=0$.

\subsection{Induced nilpotent orbits and sheets}
Given $x\in \g$ we denote by  $\mathcal{O}(x)$ the adjoint $G$-orbit of $x$. If $x$ lies in a Levi subalgebra $\l=\Lie(L)$ of $\g$ we set $\mathcal{O}_L(x):=(\Ad\,L)\cdot x$. 
\end{remark}
Combining Theorem~\ref{thm2.3} with  Proposition~\ref{prop2.4} we observe that
the unique open $G$-orbit $\mathcal{O}(e)$ in $\overline{\mathcal{D}(\l,e_0)}\cap\N(\g)$ intersects densely with
$\l(\lambda,2)_\varphi\oplus\l(\lambda,\ge 3)\oplus\n_+$
and has the property that $\dim \g_e=\dim\l_{e_0}$.
As $\l(\lambda,2)_\varphi\oplus\l(\lambda,\ge 3)\,=\,(\Ad\,P_L(\lambda))\cdot e_0$ by our choice of 
$\lambda$ we see that $\mathcal{O}(e)$
intersects densely with $\mathcal{O}_L(e_0)+\n_+$.
In other words, the orbit $\mathcal{O}(e)$ coincides with ${\rm Ind}_\l^\g\,\OO_L(e_0)$, the nilpotent $G$-orbit obtained from $\mathcal{O}_L(e_0)$ by 
Lusztig--Spaltenstein induction.

As in the characteristic zero case, induction of nilpotent orbits
is transitive in the following sense: if $L,L'$ are Levi subgroups of $G$ with $L\subset L'$ and $e_0$ is 
a nilpotent elements of $\l=\Lie(L)$ then   
\begin{equation}\label{e3}
{\rm Ind}_{\l}^{\g}\,\OO_L(e_0)\,=\,{\rm Ind}_{\l'}^{\g} \,\big({\rm Ind}_{\l}^{\l'}\,\OO_{L}(e_0)\big)
\end{equation}
where $\l'=\Lie(L')$. In order to see this it suffices to note that both orbits in (\ref{e3}) have the same dimension, the Zariski closure  of the decomposition class 
$(\Ad\,L')\cdot\big(e_0+\z(\l)_{\rm reg}\big)$ in $\l'$ contains 
$\z(\l')_{\rm reg}+e$ for some $e\in {\rm Ind}_{\l}^{\l'}\,\OO_{L}(e_0)$, and $\z(\l')\subset \z(\l)$. The statement then follows from Proposition~\ref{prop2.4}.
\begin{remark}\label{r5'} Since the decomposition class $\mathcal{D}(\l, e_0)$ is evidently independent of the
choice of a triangular decomposition of $\g$ containing $\l$, the above discussion implies
that induction of orbits is independent of the choice of a parabolic subalgebra of $\g$
containing $\l$. This was first
observed by Lusztig--Spaltenstein for unipotent classes in reductive groups over algebraically closed fields of arbitrary characteristic; see \cite[Theorem~2.2]{LSp}. The argument in {\it loc.\,cit.} relied on the theory of complex representations of
finite groups of Lie type and some results of Mizuno. More recently Lusztig returned to this subject
and found another proof of the fact that induction is independent of the choice of a parabolic. His argument is similar to ours but still works in the setting of a reductive group, possibly disconnected; see \cite[Lemma~10.3(a)]{L04}.
In principle, one could 
use a Bardsley--Richardson projection $\pi\colon\,G\to \g$ from \cite{BR} to obtain a Lie algebra analogue of  \cite[Theorem~2.2]{LSp}. Indeed, in our situation $\pi$ provides a nice $G$-equivariant isomorphism between the unipotent variety of $G$ and $\N(\g)$. However, to the best of our knowledge this line of reasoning was never carried out in the literature under our assumption on $G$ and we
felt that arguing in the spirit of \cite{BK} was more relevant for our purposes (see e.g. Remark~\ref{r4}).
\end{remark}
A nilpotent element $e_0$ in a Levi subalgebra $\l=\Lie(L)$ of $\g$ is called {\it rigid} if the orbit $\OO_L(e_0)$ cannot be obtained by Lusztig--Spaltenstein induction from a proper Levi subalgebra of $\l$.
\begin{theorem}\label{thm2.5}
If $\mathcal{S}$ is a sheet of $\g$  and $\mathcal{D}(\l,e_0)$ is the open decomposition class of $\mathcal{S}$, then $e_0$ is rigid in $\l$ and the unique nilpotent orbit in $\mathcal{S}$ has the form
$\OO={\rm Ind}_\l^\g\,\mathcal{O}_L(e_0)$. Moreover, $\dim\g_x=\dim
\l_{e_0}$ for any $x\in \OO$. Conversely, for any pair
$(\l,e_0)$, where $\l$ is a Levi subalgebra of $\g$ and $e_0$ is a rigid nilpotent element of $\l$, there is a unique sheet of $\g$ whose open decomposition class equals $\mathcal{D}(\l,e_0)$.
\end{theorem}
\begin{proof}
Let $\mathcal{D}(\l',e_0')$ be the open decomposition
class of $\mathcal{S}$ and suppose for a contradiction that $\OO_{L'}(e_0')={\rm Ind}_\l^{\l'}\,\OO_L(e_0)$ for some proper Levi subalgebra $\l$ of $\l'$ (here $L,L'$ are Levi subgroups of $G$ with $L\subset L'$ such that
$\l'=\Lie(L')$ and $\l=\Lie(L)$). Note that $\z(\l')\subset \z(\l)$ and $(\Ad\,L')\cdot (e_0+\z(\l)_{\rm reg})$ lies in $\mathcal{D}(\l,e_0)$. 
Let $\l'=\n'_-\oplus\l\oplus\n'_+$  be a triangular decomposition of $\l'$. Replacing $G$ and $\g$, respectively by $L'$ and $\l'$ in the proof of Theorem~\ref{thm2.3}
and applying (\ref{e0}) to any $z\in \z(\l)_{\rm reg}$  it is straightforward to see that $\mathcal{D}(\l,e_0)$ contains $\z(\l)_{\rm reg}+\OO_{L}(e_0)+\n'_+$ and hence $\z(\l)_{\rm reg}+e_0'$.  But then $\overline{\mathcal{D}(\l,e_0)}$ contains $\z(\l)+e_0'$. As $\z(\l')_{\rm reg}\subset \z(\l)$ it follows that  $\overline{\mathcal{S}}=\overline{\mathcal{D}(\l',e_0')}$ lies in $\overline{\mathcal{D}(\l,e_0)}$.
Combining Proposition~\ref{prop2.4} with (\ref{e3}) and our discussion at the beginning of this subsection we see that $\overline{\mathcal{D}(\l',e_0')}$ and $\overline{\mathcal{D}(\l,e_0)}$ share the same nilpotent orbit $\OO={\rm Ind}_\l^\g\,\OO_L(e_0)$.
 On the other hand,  Proposition~\ref{prop2.4} gives
\begin{eqnarray*}
\dim\mathcal{S}&=&\dim\mathcal{D}(\l',e_0')=
n-\dim\l'_{e_0'}+\dim\z(\l')=n-\dim\l_{e_0}+\dim\z(\l')\\
&<&n-\dim\l_{e_0}+\dim\z(\l)\,=\,\dim\mathcal{D}(\l,e_0)
\end{eqnarray*}
contrary to the fact that $\mathcal{S}$ is an irreducible component of $\g_{(m)}$ where $m=n-\dim\l_{e_0}$. This contradiction shows that $\l'=\l$ and $e_0'=e_0$ is rigid in $\l$.  

If $\OO$ is the nilpotent orbit of $\mathcal{S}$ 
then $\dim\OO=n-\l_{e_0}$ by Proposition~\ref{prop2.4}.
Then $\dim\g_x=\dim\l_{e_0}$ for all $x\in\OO$ because $\g_x=\Lie(G_x)$.

Finally, suppose $e_0'$ is a rigid nilpotent element in a Levi subalgebra $\l'$ of $\g$ and put $s=\dim\l'_{e'_0}$. Let $\mathcal{S}'$ be the union of all adjoint $G$-orbits of dimension $n-s$ contained in $\overline{\mathcal{D}(\l',e_0')}$. It follows from Chevalley's semi-continuity theorem that $\mathcal{S}'$ is a Zariski open subset of $\overline{\mathcal{D}(\l',e_0')}$ (see the end of the proof of Proposition~\ref{prop2.4} for a similar argument). Since $\overline{\mathcal{D}(\l',e_0')}$ is irreducible, so is the quasi-affine variety $\mathcal{S}'$. We claim that $\mathcal{S}'$ is a sheet of $\g$. Indeed, suppose the contrary. Then $\mathcal{S}'\subsetneq \tilde{\mathcal{S}}$ for some sheet $\tilde{\mathcal{S}}$ of $\g$ contained in $\g_{(n-s)}$. Let $\mathcal{D}(\l,e_0)$ be the open decomposition class of $\tilde{\mathcal{S}}$. We may assume without loss of generality that both $\l'$ and $\l$ are standard Levi subalgebras of $\g$, so that the maximal toral subalgebra $\t=\Lie(T)\subseteq \l\cap\l'$ of $\g$ contains both $\z(\l)$ and $\z(\l')$.

Let $z\in\z(\l')_{\rm reg}$. Since $\tilde{\mathcal{S}}$ contains $\mathcal{D}(\l',e_0')$ and the Zariski closure of $\tilde{\mathcal{S}}$ coincides with $\overline{\mathcal{D}(\l,e_0)}$ which, in turn, lies in $ \overline{\g_{(n-s)}}$, 
the $G$-orbit $\OO(z+e_0')$ is open in one of the fibres of
the morphism $\bar{F}\colon\,\overline{\mathcal{D}(\l,e_0)}\to \t/W$. In view of 
(\ref{e1}) it coincides with $(\Ad\,G)\cdot \big(t+V)$ for some $t\in (W\cdot z)\cap \z(\l)$ and some large subset $V$ of $\l(\lambda,\ge 2)+\n_+\cap\g_t$ which we are now going to describe. Since both $\l$ and $\n_+\cap \g_t$ are centralised by $t$, combining (\ref{e1}) with the inclusion $\overline{\mathcal{D}(\l,e_0)}\subseteq \overline{\g_{(n-s)}}$ one observes that $V$  consists of all elements in $\l(\lambda,\ge 2)+\n_+\cap\g_t$ whose centraliser in $\g_t$ has dimension $s$. Furthermore, $V$ is Zariski open in
$\l(\lambda,\ge 2)+\n_+\cap\g_t$

Write $z+e_0'=(\Ad\,g)\cdot(t+v)$ where $g\in G$ and $v\in V$ and choose $n_w\in N_G(T)$ such that $(\Ad\, n_w)\cdot z=t$. Since $[t,v]=0=[z,e_0']$ we have that
$(\Ad\,g)\cdot t=z$ and $(\Ad\, g)\cdot v=e_0'$. Then
$gn_w\in G_z$. Since $z\in\z(\l')_{\rm reg}$, and $G$ satisfies the standard hypotheses, applying \cite[3.19, 4.2]{SpSt} yields $G_z=L'$. 
Set $v':=(\Ad\,n_w^{-1})\cdot v$. Since $e_0'=(\Ad\, gn_w)\cdot v'$, we now deduce that
$v'\in \OO_{L'}(e_0')$. On the other hand, the concluding remark of the preceding paragraph in conjunction with (\ref{e1}) shows that $v$ lies in the $G_t$-orbit which intersects densely with $\l(\lambda,\ge 2)+\n_+
\cap\g_t\subset \g_t$. Since $L'=n_w^{-1} G_t n_w$, it follows that $\OO_{L'}(e_0')$ intersects densely with 
 $(\Ad\,n_w^{-1})\cdot\big(\l(\lambda,\ge 2)+\n_+\cap\g_t)\big)\subset \g_z=\l'$. 
 
Since the orbit $\OO_{L'}(e_0')$ is rigid in $\l'$ this yields
$\n_+\cap \g_t=0$ (otherwise $\OO_{L'}(e_0')$ would be induced from $(\Ad\,n_w^{-1})(\l)$, a proper Levi subalgebra of $\l'$). Since $t\in\z(\l)$ and $\g_t=(\n_{-}\cap\g_t)\oplus \l\oplus (\n_+\cap\g_t)$ we now obtain that $t\in\z(\l)_{\rm reg}$. Since $t\in W\cdot z$ and $G_z=L'$ it follows that 
$\dim\z(\l)=\dim\z(\l')$ and
$$\dim\mathcal{S}'=\dim\mathcal{D}(\l',e_0')=
n-s+\dim\z(\l')=n-s+\dim\z(\l)=\dim\mathcal{D}(\l,e_0)=\dim\tilde{\mathcal{S}}$$
by Proposition~\ref{prop2.4}. This contradiction shows that $\mathcal{S}'=\tilde{S}$ is a sheet of $\g$ thereby completing the proof.
\end{proof}
\begin{remark}\label{r6}
Theorem~\ref{thm2.5} implies that the sheets of $\g$ are in 1--1 correspondence with the $G$-conjugacy classes of pairs $(L,\OO_L)$ where $L$ is a Levi subgroup of $G$ and $\OO_L$ is a rigid nilpotent $L$-orbit in 
$\Lie(L)$. In the characteristic zero case, this is a well  known result of Borho; see \cite{Borho}. 
\end{remark}
\section{Classifying the rigid orbits in exceptional Lie algebras}
\subsection{Some historical remarks} 
In order to complete the picture outlined in Section~\ref{Sec2} one needs to classify all rigid orbits in the reductive Lie algebras $\g=\Lie(G)$. Of course, the general case reduces quickly to the case where the group $G$ is simple, and here the problem has a long history.
All rigid nilpotent orbits in Lie algebras of classical groups are described by Kraft \cite{Kraft} in type ${\rm A}$ and by Kempken \cite{Ke} and Spaltenstein \cite{Sp-book} in types ${\rm B,C,D}$. 
If $\g=\sl_n$ then $\{0\}$ is the only rigid orbit in $\g$, whilst for $\g=\so_n$ or $\sp_{n}$  
the classification is given in terms of partitions of $n$. 
More precisely, it is proved in \cite{Kraft} that a nilpotent element in $\sl_n$ corresponding to a partition $\lambda$ of $n$ is Richardson
in a parabolic subalgebra of $\sl_n$
associated with the dual partition $\lambda^t$.
The arguments in \cite{Ke} rely on Kraft's result
in a crucial way.

Although the main theorems in {\it loc.\,cit.} are stated under the assumption that the base field has characteristic $0$, the proofs are based on linear algebra and the same combinatorial description is valid in any good characteristic, i.e. for $p\ne 2$.
In characteristic $2$, nilpotent orbits and sheets in $\so_n$ and $\sp_n$ were studied by Hesselink \cite{He} and by Spaltenstein \cite{Sp1} who 
conjectured that every sheet in a reductive Lie algebra over an algebraically closed field of arbitrary characteristic should contain a {\it unique} nilpotent orbit. Of course, confirming this would imply that
the number of nilpotent orbits in $\g$ is finite.
The main results of \cite{Sp1}
show that the conjecture holds for $G$ classical, but the general case remains open in bad characteristic.
We mention that Spaltenstein's conjecture makes sense in the context of infinitesimal sheets 
and their coadjoint analogues (see Remark~\ref{r1}) and confirming it would enable one to
prove a natural analogue of the
Kac--Weisfeiler conjecture in arbitrary characteristic; see \cite[5.6]{PSk} for detail.

For $G$ exceptional, the rigid orbits in $\N(\g)$ were first 
classified by Elashvili \cite{Elashvili} under the assumption that ${\rm char}(k)=0$. Furthermore, for any orbit $\OO\subset\N(\g)$, Elashvili lists
all sheets of $\g$ containing $\OO$. This important result is presented in {\it loc.\,cit.} in the form of tables which were recently double-checked 
in \cite{dG-E} by using computational methods. 
\subsection{Odd nilpotent elements} It was first observed by Lusztig--Spaltenstein that if ${\rm char}(k)$ is $0$ or large for $G$ then many nilpotent orbits in $\g$ can be induced in a very natural way; see \cite[Prop.~1.9]{LSp}. 
The goal of this subsection it to show that this result is still valid under the assumption that $G$ satisfies the standard hypotheses.

Let $\Phi$ be the root system of $G$ with respect to a maximal torus $T$ of $G$ and let $\Pi$ be a basis of simple roots in $\Phi$. 
Each orbit $\OO\subset \N(\g)$ is uniquely labelled by its weighted Dynkin diagram $D$ which assigns to each $\alpha\in\Pi$
an integer $D(\alpha)\in\{0,1,2\}$. According to \cite[2.3, 2.4]{P03}, each such $D$ gives rise to a cocharacter $\lambda=\lambda_D\in X_*\big(T\cap\mathcal {D}G\big)$ which is optimal in the sense of the Kempf--Rousseau theory for every element of a principal Zariski open subset of $\g(\lambda,2)$. That subset, denoted earlier by $\g(\lambda,2)_\varphi$, coincides with the intersection of $\OO$ with $\g(\lambda,2)$. 
A nilpotent orbit $\OO=\OO(D)$ is said to be {\it odd} if $D(\alpha)\in\{0,1\}$ for all $\alpha\in\Pi$.
 
Given a weighted Dynkin diagram $D$ we let $\Pi'=\Pi'(D)$
be the set of all $\alpha\in\Pi$ such that $D(\alpha)\in\{0,1\}$ and denote by $L=L(D)$ the
standard Levi subgroup of $G$ generated by $T$ and all unipotent root subgroups $U_\alpha$ with $\alpha\in\Pi'$. 
Given a nonempty subset $\Pi_0\subseteq\Pi$ and a weighted Dynkin diagram $D_0$ of the standard Levi subalgebra of $\g$ associated with $\Pi_0$ we define $\widetilde{D}_0\colon\,\Pi\to \{0,1,2\}$ by setting $\widetilde{D}_0(\alpha)=D_0(\alpha)$
for all $\alpha\in \Pi_0$ and $\widetilde{D}_0(\alpha)=2$ for all $\alpha\in\Pi\setminus\Pi_0$.
\begin{prop}\label{prop3.2} The following are true:

(i) If the orbit $\OO=\OO(D)$ is not odd then $\l=\Lie(L(D))$ is a proper Levi subalgebra of $\g$ and $\OO(D)\,=\,{\rm Ind}_{\l}^\g
\,\OO_{L}(e_0)$ for some nilpotent element  $e_0\in\l$.

(ii) Let $L'$ be the standard Levi
subgroup of $G$ associated with a nonempty subset $\Pi_0$ of $\Pi$ and let
 $D_0\colon\,\Pi_0\to\{0,1,2\}$ be a weighted Dynkin diagram of $\l'=\Lie(L')$.
If $\widetilde{D}_0$ is a valid weighted Dynkin diagram for $\g$ then 
$\OO(\widetilde{D}_0)={\rm Ind}_{\l'}^\g\,\OO_{L'}(D_0)$.
\end{prop}
\begin{proof} It is clear that if $D(\alpha)=2$ for some $\alpha\in\Pi$ then $\l$ is a proper Levi subalgebra of $\g$. Set $\Pi(j):=\{\alpha\in\Pi\,|\,\,D(\alpha)=j\}$
where $j\in\{0,1,2\}$.
If $\Pi=\{\alpha_1,\ldots,\alpha_\ell\}$
and $\beta=\sum_{i=1}^\ell m_i\alpha_i$, where $m_i\in \Z$, then we put $\nu_i(\beta):=m_i$ for all $1\le i\le \ell$.
From the definition of $\lambda=\lambda_D$ it follows that 
$\g(\lambda,i)=\l(\lambda,i)$ for $i=0,1$ and
$\g(\lambda,\pm 2)=\l(\lambda,\pm 2)\oplus \m(\lambda,\pm 2)$ where $\m(\lambda,\pm 2)$ is the 
$k$-span of all root vectors $e_\beta\in \Lie(U_\beta)$ with $\nu_i(\beta)=1$ for $i\in \Pi(2)$ and $\nu(\beta)=0$ for $i\in \Pi(1)$. 

Let $e\in\g(2,\lambda)_\varphi$ and write $e=e_0+e_1$ with $e_0\in\l(\lambda,2)$ and $e_1\in \m(\lambda,2)$.
Since 
$\g_e\subset \g(\lambda,\ge 0)$ and $[e,\g(\lambda,\ge 0)]=\g(\lambda,\ge 2)$ by \cite[Theorem~2.3]{P03}, we have that $$\dim\g_e=\dim\g(\lambda,0)+\dim\g(\lambda,1)=\dim\l(\lambda,0)+
\dim\l(\lambda,1).$$
Therefore, in order to prove the first part of the proposition it suffices to show that $\dim\l_{e_0}=\dim\l(\lambda,0)+\dim\l(\lambda,1)$.
Indeed, in that case the orbit ${\rm Ind}_\l^\g\,\OO_L(e_0)$ would intersect densely with $\g(\lambda,\ge 2)$ and hence coincide with $\OO(D)$.

 Since $\Phi=\Phi_+(\Pi)\sqcup \Phi_-(\Pi)$ it is straightforward to see that
$\big[\l(\lambda, -i),\m(\lambda,2)\big]=0$ for all $i\in\Z_{>0}.$ 
It follows that the restrictions of $\ad\,e$ and $\ad\,e_0$ to $\l(\lambda,-i)$ coincide for all such $i$. Since $\g_e\subset \g(\lambda,\ge 0)$ by \cite[Theorem~A]{P03} this entails $\l_{e_0}\subseteq \l(\lambda,\ge 0)$.

It is well known (and easily seen) that the restriction of $\kappa$ to $\l$ is non-degenerate
and the subspaces $\l(\lambda,\pm i)$ are dual to each other with respect to $\kappa$. 
Since $\l_{e_0}$ coincides with the orthogonal complement of $[e_0,\l]$ with respect to the restriction of
$\kappa$ to $\l$ and $\l_{e_0}\subseteq \l(\lambda,\ge 0)$ by the above, 
the map $\ad\,e_0\colon\,\l(\lambda,i)\to \l(\lambda,i+2)$ must be surjective for all $i\in\Z_{\ge 0}$.
From this it is immediate that $\dim\l_{e_0}=\dim\l(\lambda,0)+\dim\l(\lambda,1)$  proving (i).

To prove (ii), we let $\lambda'$ and $\lambda$ be the cocharacters in $X_*(T)$ attached to $D_0$ and $\widetilde{D}_0$ at the beginning of this subsection. It is immediate from the definition of $\widetilde{D}_0$ that
$\l'(\lambda',2)\subseteq \g(\lambda,2)$ and
$\g(\lambda,i)\cap\l'=\l'(\lambda',i)$ for $i=1,2$.  
Let $V$ be the $T$-stable subspace of $\g(\lambda,2)$ complementary to $\l'(\lambda',2)$.
Being an open map the first
projection ${\rm pr_1}\colon\,\g(\lambda,2)=\l'(\lambda',2)\oplus V\twoheadrightarrow \l'(\lambda',2)$ takes the Zariski open subset $\g(\lambda,2)_\varphi$ of $\g(\lambda,2)$ onto a dense subset of $\l'(\lambda',2)$. 
In view of \cite[Theorem~2.3]{P03} this implies that there exists $e'\in\OO_{L'}(D_0)\cap \l'(\lambda',2)$ such that $e'+v\in \g(\lambda,2)_\varphi$ for some 
$v\in V$. As $\widetilde{D}_0$ is a weighted Dynkin diagram for $\g$, {\it loc.\,cit.} also shows that the cocharacter $\lambda$ is optimal for $e:=e'+v$ and 
$$
\dim \g_e=\dim\g(\lambda,0)+\dim\g(\lambda,1)=
\dim\l'(\lambda',0)+\dim\l'(\lambda',1)=
\dim{\l'}_{e'}.
$$ Since it is straightforward to see that there exists a standard parabolic subgroup $P=L'R_u(P)$ of $G$ such that
$V\subseteq \Lie(R_u(P))$, we now conclude that
$e=e'+v\in{\rm Ind}_{\l'}^\g\,\OO_{L'}(D_0)$. This  finishes the proof.
\end{proof}
\begin{remark}
For $G$ exceptional, the second part of Proposition~\ref{prop3.2} is an immediate consequence of 
\cite[Theorem~3(i)]{LT11} which was proved by computational methods.
\end{remark}
\subsection{A sufficient condition for rigidity}
Given a nilpotent element $e\in\g$ we denote by
$\c_e$ the factor space $\g_e/[\g_e,\g_e]$ and
set $c(\g_e):=\dim\c_e$. Our next result is a  generalisation of \cite[Proposition~11]{Ya}.
\begin{prop}\label{deform}
Assume, as before, that $G$ satisfies the standard hypotheses and 
let $\l=\Lie(L)$ be a Levi subalgebra of $\g$. If $e\in{\rm Ind}_\l^\g\,\OO_L(e_0)$
for some $e_0\in\N(\l)$
then 
$c(\g_e)\ge c(\l_{e_0})$ and $\dim\z(\g_e)\ge \dim\z(\l_{e_0})$.
Furthermore, if $L$ has no simple components isomorphic to $\SL_{rp}$, where $p={\rm char}(k)$, then $c(\g_e)\ge c\big([\l,\l]_{e_0}\big)+\dim\z(\l)$ and $\dim \z(\g_e)\ge \dim\z(\l)+\dim \z\big([\l,\l]_{e_0}\big).$ 
\end{prop}
\begin{proof}
Put $\g[t]:=\g\otimes_k k[t]$ and $\g(t):=\g\otimes_k k(t)$, where $t$ is an indeterminate, and let 
$\pi\colon\,\g[t]\to \g$ denote the canonical projection induced by the augmentation map $k[t]\to k$. 
Let $P=L\cdot R_u(P)$ be a parabolic subgroup of $G$ containing $L$ and $\n_+:=\Lie(R_u(P))$. By our assumption, $\dim\g_e=\dim\l_{e_0}$ and we may assume without loss of generality that $e=e_0+e_1$ for some $e_1\in\n_+$.

Let $h\in \z(\l)_{\rm reg}$ and consider
$th+e$, an element of $\g[t]$. Let $G(t)$ be the group
of the $k(t)$-rational points of $G$.
The argument used in the proof of  Theorem~\ref{thm2.3} then shows that $th+e$ is 
$G(t)$-conjugate to $th+e_0$ (in fact, a conjugating element can be found already in the group of $k(t)$-rational points of $R_u(P)$). 
From this it is immediate that $\g(t)_{th+e}\cong \l_{e_0}\otimes_k k(t)$ as Lie algebras over $k(t)$.
It follows that
\begin{eqnarray*}
\dim_{k(t)}\big[\g(t)_{th+e},\g(t)_{th+e}\big]&=&\dim[\l_{e_0},\l_{e_0}],\\
\dim_{k(t)}\,\z\big(\g(t)_{th+e}\big)&=&\dim\z(\l_{e_0}).
\end{eqnarray*}
Furthermore, if $\l$ has no components of type ${\rm A}_{rp-1}$,  our discussion in \S~\ref{SH} implies that $\dim[\l_{e_0},\l_{e_0}]=\dim [\l,\l]_{e_0}$ and 
$\dim\z(\l_{e_0})=\dim\z(\l)+\dim
\z\big([\l,\l]_{e_0}\big)$.

Put $M_1:=\g(t)_{th+e}\cap\g[t]$, $M_2:=\big[\g(t)_{th+e},\g(t)_{th+e}\big]\cap\g[t]$ and $M_3:=\z\big(\g(t)_{th+e}\big)\cap\g[t]$. Each of these is a $k[t]$-submodule of the free $k[t]$-module $\g[t]$, and it is not hard to check that all factor modules $\g[t]/M_i$ are torsion-free.
Since $k[t]$ is a principal ideal domain and the $k[t]$-module $\g[t]$ is finitely generated, it follows that all $\g[t]/M_i$ are free over $k[t]$. As a consequence, there exist free $k[t]$-submodules $N_i$ of $\g[t]$ such that  $\g[t]=N_i\oplus M_i$ as $k[t]$-modules, where $i=1,2,3$. This, in turn, 
shows that each $M_i$ is free over $k[t]$ and
$\rk_{k[t]}M_i=\dim_{k(t)} M_i\otimes_{k[t]}\,k(t)$. 

It is straightforward to see that $\pi\colon\,\g[t]\to \g$ maps $M_1$ into $\g_e$.
Since $e\in {\rm Ind}_{\l}^\g\,\OO_L(e_0)$ we have that $$\dim\g_e=\dim \l_{e_0}=\dim_{k(t)}\g(t)_{th+e}=\dim_{k(t)} M_1\otimes_{k[t]} k(t)=\rk_{k[t]} M_1.$$
From this it follows that $M_1$ has a free basis $v_1(t),\ldots,v_r(t)$ such that $\pi(v_1(t)),\ldots,\pi(v_r(t))$ form a $k$-basis of $\g_e$. For $1\le i\le r$ put $v_i:=\pi(v_i(t))$ and observe that
$$\rk_{k[t]}M_2=\dim M_2\otimes_{k[t]}k(t)=\dim_{k(t)}
\big[\g(t)_{th+e},\g(t)_{th+e}\big]=\dim[\l_{e_0},\l_{e_0}].$$ Since
$[v_i(t),v_j(t)]\in M_2$ and $\pi\big([v_i(t),v_j(t)]\big)=[v_i,v_j]$ for $1\le i,j\le r$, this gives 
$\dim [\g_e,\g_e]\le \dim [\l_{e_0},\l_{e_0}]$. Therefore, $c(\g_e)=\dim\g_e-\dim[\g_e,\g_e]\ge \dim \l_{e_0}-\dim[\l_{e_0},\l_{e_0}]=c(\l_{e_0}).$

Finally, we observe that $\pi(M_3)\subseteq \z(\g_e)$
and $$\rk_{k[t]} M_3=\dim_{k(t)}M_3\otimes_{k[t]}k(t)=\dim _{k(t)}\z\big(\g(t)_{th+e}\big)
=\dim\z(\l_{e_0}).$$
Consequently, $\dim\z(\g_e)\ge \dim\z(\l_{e_0})$. The second part of the proposition now follows from our remarks earlier in the proof.
\end{proof}
Proposition~\ref{deform} gives us a useful sufficient condition for rigidity of nilpotent elements in exceptional Lie algebras.
\begin{cor}\label{almostrigid} Suppose $G$ is simple and $p$ is good for $G$. 
If $e\in\N(\g)$ is such that $\g_e=[\g_e,\g_e]$ then $e$ is rigid in $\g$.
\end{cor}
\begin{proof}
We may assume without loss that $G$ is simply connected. 
For $G$ classical the statement follows by comparing \cite[Theorem~3]{PT14} with the Kempken--Spaltenstein description of rigid nilpotent nilpotent orbits in $\g$. (Of course,
in the characteristic zero case one can apply 
\cite[Proposition~11]{Ya} directly.)

So suppose from now that $G$ is exceptional and $p$ is a good prime for $G$. Since the Killing form of $\g$ is non-degenerate, $G$ satisfies the standard hypotheses.  By rank considerations, there are no Levi subgroups in $G$ with components of type ${\rm A}_{rp-1}$ for $r\ge 2$. 

Suppose
$e\in{\rm Ind}_\l^\g\,\OO_L(e_0)$ where $L$ is a proper Levi subgroup of $G$ and $e_0\in\N(\l)$.
Since $c(\g_e)=0$ in the present case, Proposition~\ref{deform} yields that $L$ has a component of type ${\rm A}_{p-1}$. Then all components of $L$ have type ${\rm A}$. Let $T$ be a maximal torus of $L$ and $\t=\Lie(T)$. We may assume that $L$ is associated with a subset $\Pi_0$ of a basis of simple roots $\Pi$
 of the root system $\Phi(G,T)$. Since all components of $L$ have type ${\rm A}$ we may also assume without loss that $e_0=\sum_{\alpha\in \Pi_1} e_\alpha$ for some subset $\Pi_1$ of $\Pi_0$ (possibly empty). 
Let $\t_0:=\t\cap[\l,\l]$ and $\t_1:=
 \bigcap_{\alpha\in\Pi_1}\ker\,({\rm d}\alpha)_e$.  Then $\t_0\cap\t_1$ consists of all $x\in\t_0$ such that $({\rm d}\alpha)_e(x)=0$ for all $\alpha\in\Pi_1$. 
 Since all simple components of $L$ are simply connected, looking at the $p$-ranks of the Cartan matrices associated with  $\Pi$ and $\Pi_0$ one observes that $\dim \t_0=|\Pi_0|$, $\dim \t_1=\dim\t-|\Pi_1|$ and
 $\dim(\t_0\cap\t_1)=|\Pi_0|-|\Pi_1|$ unless one of the components of
 $\Pi_1$ has type ${\rm A}_{p-1}$ in which case $\dim(\t_0\cap\t_1)=|\Pi_0|-|\Pi_1|+1$.
 
If $\Pi_1$ does not contain components of type ${\rm A}_{p-1}$ the above shows that $\dim\t_1=
|\Pi|-|\Pi_1|>|\Pi_0|-|\Pi_1|=\dim\t_0\cap\t_1.$
Hence there exists $h\in \t$ such that $h\not\in [\l,\l]$ and $[h,e_0]=0$. But then $\l_{e_0}\supsetneq [\l_{e_0},\l_{e_0}]$, so that $c(\l_{e_0})\ge 1$.
Since this contradicts Proposition~\ref{deform}, we see that $\Pi_1$ must have a component of type ${\rm A}_{p-1}$. Since $\Pi_0$ has such a component, too, they must contain the same component of type $A_{p-1}$. Let $L_1$ be the simple component of $L$ generated by the root subgroups $U_{\pm \alpha}$ with $\alpha\in\Pi_1$. Then $L_1\cong \SL_p$ as algebraic groups. The Lie algebra $\l$ acts on its ideal $\l_1:=\Lie(L_1)$ by derivations, giving a natural Lie algebra homomorphism
$\psi\colon\,\l\to\Der(\l_1)$. Since $\l_1\cong\sl_p$ we may identify $\Der(\l_1)$ with $\mathfrak{pgl}_p$. Our earlier remarks now imply that $\psi(e_0)$ is a regular nilpotent element in $\mathfrak{pgl}_p$ and hence has an abelian centraliser in $\mathfrak{pgl}_p$.
As $\psi$ maps $\l_{e_0}$ onto a nonzero subalgebra of the centraliser
of $\psi(e_0)$ in $\mathfrak{pgl}_p$, the Lie algebra $\l_{e_0}$ has a nontrivial abelian quotient. But then
$\l_{e_0}$ is not perfect, i.e. $c(\l_{e_0})\ge 1$.
Since this contradicts Proposition~\ref{deform}, the element $e$ is rigid in $\g$.
\end{proof}
\subsection{Optimal cocharacters of nilpotent elements contained in regular subalgebras} In this subsection we assume that $G$ is an arbitrary connected reductive group defined over an algebraically closed field of good characteristic. A connected reductive subgroup $K$ of $G$ is called {\it regular} if it contains a maximal torus of $G$. We denote by $\k$ the Lie algebra of $K$.
If $T$ is a maximal torus of $G$ contained in $K$ then the root system $\Phi_0$ of $K$ with respect to $T$ identifies with a root subsystem of the root system $\Phi=\Phi(G,T)$ and 
$\k=\t\oplus\sum_{\alpha\in\Phi_0}\g_\alpha$ where $\t=\Lie(T)$. Since $\k$ is a restricted Lie subalgebra of $\g$ any nilpotent element of $\k$ is contained in $\N(\g)$, the nilpotent cone of $\g$. 

If ${\rm char}(k)=0$ then a standard argument involving $\sl_2$-triples shows that
any nonzero nilpotent element $e\in\k$ admits a rational cocharacter $\lambda\colon k^\times\to G$ 
which is optimal for $e$ in the sense of the 
Kempf--Rousseau theory and has the property that $\lambda(k^\times)\subset K$. 
Our next goal is to show that this result holds
in under our assumption on $k$. 
Given a Zariski closed subgroup $H$ of $G$ we write
 $X_*(H)$ for the set of all rational cocharacters $\lambda\colon\, k^\times \to H$ and we denote by $P_H(\lambda)$ the parabolic subgroup of $H$ associated with $\lambda\in X_*(H)$.
If $e$ is a nonzero nilpotent element of $\k$  then we denote by
$\hat{\Lambda}_K(e)$ (resp. $\hat{\Lambda}_G(e)$) the set of all $\lambda\in X_*(K)$ (resp. $\lambda\in X_*(G))$ which are optimal for $e$ regarded as a $K$-unstable element
of $\k$ (resp. a $G$ unstable element of $\g$).
\begin{prop}\label{regular} If $K$ is a regular reductive subgroup of $G$ and $p={\rm char}(k)$ is a good prime for $G$ then 
for any nonzero nilpotent element of $e\in \k$ we have the inclusion $\hat{\Lambda}_K(e)\subseteq \hat{\Lambda}_G(e)$. 
\end{prop}
\begin{proof} It proving this proposition we may assume without loss of generality that $K$ is a maximal connected reductive subgroup of $G$.  
Thanks to\cite[Lemma~14]{McN04} we may also assume that
$G$ is a semisimple group of adjoint type. Let $\Phi$ be the root system of $G$ with respect to a maximal torus $T$ contained in $K$ and let $\Phi_0$ be the set of roots of $K$ with respect to $T$. The maximality of $K$ then shows that $\Phi_0$ is a maximal root subsystem of $\Phi$. Choose a basis of simple roots $\Pi$ in $\Phi$ and let $\tilde{\alpha}_1, \ldots, \tilde{\alpha}_s$
 be all highest roots of $\Phi$ with respect to $\Pi$
(here $s$ is the number of the simple components of $G$). For $1\le i\le s$ write $\tilde{\alpha}_i=\sum_{\alpha\in \Pi}\,n(\alpha,i) \alpha$. In view of the Borel--de Siebenthal theorem, no generality will be lost by assuming that either $K$ is a maximal standard Levi subgroup of $G$ or there exist $\alpha_0\in \Pi$ and a prime number $q$ dividing $n(\alpha_0,i_0)$, where $i_0$ is the unique index $j\le s$ for which $n(\alpha_0,j)\ne 0$, such that $\Phi_0=\{\gamma=\sum_{\alpha\in\Pi}\,
 r_\gamma(\alpha)\alpha\in\Phi\,|\,\,q\mbox{ divides } r_\gamma(\alpha_0)\}$. 
 
It is well known that there is a canonical duality
$\langle\cdot\,,\,\cdot\rangle$ between $X_*(T)$ and the lattice of rational characters $X^*(T)$. Let $\{\varpi^\vee_\alpha\,|\,\,\alpha\in \Pi\}\subset X_*(T)$  be the system of fundamental coweights corresponding to $\Pi\subset X^*(T)$. It has the property that 
$\langle\varpi^\vee_\alpha,\beta \rangle=\delta_{\alpha,\beta}$ for all $\alpha,\beta\in \Pi$. 
Since $G$ is a group of adjoint type, 
for every $\alpha\in\Pi$ there is  
a $1$-parameter subgroup $\varpi^\vee_\alpha(k^\times)$ in $T$ such that
\begin{eqnarray}\label{omega}\big(\Ad\,\varpi_\alpha^\vee(t)\big)(e_\gamma)=t^{r_\gamma(\alpha)}e_\gamma\,\qquad\ \big(\forall\gamma\in\Phi,\,\forall
e_\gamma\in \g_\gamma,\,\forall\,t\in k^\times\big).\end{eqnarray} 
Since $p$ is a good prime for $G$ we have that $q\ne {\rm char}(k)$. Hence $k^\times$ contains a $q$-th primitive root of $1$. Thanks to (\ref{omega}) and the preceding discussion that this implies that $T$
contains an element $\sigma$ such that $\k$ coincides with $\g^\sigma$, the Lie algebra of the fixed point group $G^\sigma$. 

Let $\lambda_0\in \hat{\Lambda}_K(e)$ and $\lambda\in\hat{\Lambda}_G(e)$. As $\k=\g^\sigma$ and $G_e\subset P_G(\lambda)$ by the optimality of $P_G(\lambda)$ it must be that $\sigma\in P_G(\lambda)$.
As $T\subset K$, the element $\sigma$ lies in the centre of $K$.
As $K$ is connected, the latter is contained in any parabolic subgroup of $K$. So
$\sigma\in P_K(\lambda_0)\subset P_G(\lambda_0)$. But then
$\sigma$ belongs to the connected group $P_G(\lambda)\cap P_G(\lambda_0)$ and therefore lies in a
maximal torus of $P_G(\lambda)\cap P_G(\lambda_0)$; we call it $T'$. Note that $T'\subset K$ because $T'$ is a connected group commuting with $\sigma$. 

Since $T'$ is a maximal torus of $P_G(\lambda)$, the set $X_*(T')\cap \hat{\Lambda}_G(e)$ is nonempty; see \cite[Theorem~2.1(iii)]{P03}, for example.
We claim that $X_*(T')\cap \hat{\Lambda}_G(e)\subset\hat{\Lambda}_K(e)$. To prove this
we adopt the notation and conventions of \cite[2.2]{P03}
and recall that there exists $\lambda'\in X_*(T')\cap \hat{\Lambda}_G(e)$ for which
$m(\lambda',e)=2$; see \cite[Theorems~2.3(i) and 2.7]{P03}.
Since $\lambda'\in\hat{\Lambda}_G(e)$ we have that
$$\frac{m(\lambda',e)}{\|\lambda'\|}\,\ge\,\frac{m(\mu,e)}{\|\mu\|}\qquad\ \big(\forall\,\mu\in \hat{\Lambda}_K(e)\big).$$
But it also follows from {\it loc.\,cit.} that there is a $\mu'\in\hat{\Lambda}_K(e)$ for which $m(\mu',e)=2$ (here we regard $e$ as an element of $\k$).
From this it is immediate that $\lambda'\in\hat{\Lambda}_K(e)$. Since the quantity $m(\mu,e)/\|\mu\|$ is independent of the choice of $\mu\in\hat{\Lambda}_K(e)$ the claim follows and completes the proof.
\end{proof}
Working over $\C$ de Graaf and Elashvili determine the weighted Dynkin diagram of the unique nilpotent orbit contained in a sheet $\mathcal{S}$ of an exceptional Lie algebra $\g_\C$ and give a nice representative $e_{\Gamma,\C}=\sum_{\gamma\in\Gamma}\,e_{\gamma,\C}$ in that orbit. 
Here $\Gamma=\Gamma(\mathcal{S})$ is a subset of roots the root system $\Phi$ of $\g_\C$ and $e_{\gamma,\C}$ is a root vector of $\g_\C$ corresponding to $\gamma\in\Phi$. 
Since each set $\Gamma$ consists of linearly independent roots, the $G_\C$-orbit of $e_{\Gamma,\C}$ is independent of the choices of root vectors
$e_{\gamma,\C}$ (in the sense that each of them can be rescaled). In \cite{dG-E}, the roots in each set $\Gamma$ are labelled by their positions in the ordering of positive roots used in GAP4\footnote{We are thankful to Simon Goodwin for bringing this to our attention}. 

To each set $\Gamma$ de Graaf and Elashvili assign a diagram $D(\Gamma)$. Specifically, the number of nodes of $D(\Gamma)$ equals the cardinality of $\Gamma$ and the nodes depicting distinct $\gamma,\gamma'\in \Gamma$ are linked by $\langle \gamma,\gamma'\rangle\cdot \langle \gamma',\gamma\rangle
$ edges. The edges are solid (resp. dotted) if $\langle \gamma,\gamma'\rangle$ is negative (resp. positive).
If $\Phi$ has roots of two different lengths then the nodes of $D(\Gamma)$ corresponding to long roots
of $\Phi$ are coloured in black. 

Since $\g=\Lie(G)$ contains a natural analogue $e_\Gamma$ of $e_{\Gamma,\C}$ we wonder whether the nilpotent orbits of $e_\Gamma$ and $e_{\Gamma,\C}$ have the same weighted Dynkin diagram. In the next section we give a positive answer to this question by computational methods, but our next result indicates that in many cases such computations can be avoided.
\begin{cor}\label{princ}
Suppose that $p$ is a good prime for $G$, all roots in $\Phi$ have the same length, and $\Gamma$ is such that 
$\gamma-\gamma'\not\in\Phi$ for all distinct $\gamma,\gamma'\in\Gamma$. If 
$D(\Gamma)$ is a disjoint union of Dynkin graphs then the nilpotent orbits in $\g$ and $\g_\C$ containing $e_\Gamma\in \g$ and $e_{\Gamma,\C}$, respectively,
have the same labels and dimensions.
\end{cor}
 \begin{proof} 
We may assume that 
$\Phi$ is the root system of $G$ with respect to a maximal torus $T$ of $G$ and $\Gamma$ is contained in the positive system $\Phi_+(\Pi)$ associated with a basis of simple roots $\Pi$ of $\Phi$. 
Let $W_0$ be the subgroup of the Weyl group $W(\Phi)$
generated by all reflections $s_\gamma$ with $\gamma\in\Gamma$ and put $\Phi(\Gamma):=\{w(\gamma)\,|\,\,w\in W_0,,\gamma\in\Gamma\}$. Using our assumptions on $\Gamma$ it is straightforward to see that $\Phi(\Gamma)=\Phi_+(\Gamma)
\sqcup -\Phi_+(\Gamma)$ where $\Phi_+(\Gamma)$ consists of roots in $\Phi$ which can be presented as
linear combinations of of elements in $\Gamma$ with coefficients in $\Z_{\ge 0}$. 
If follows
that $\Phi(\Gamma)$ is a root system in the $\Q$-span of $\Gamma$ and $\Gamma$ is a basis of simple roots in $\Phi(\Gamma)$.

Let $K$ be be the connected subgroup of $G$ generated by $T$ and all root subgroups $U_{\pm\gamma}$ with $\gamma\in\Phi(\Gamma)$. 
The preceding discussion shows that the quadruple
$\big(X^*(T),\Phi(\Gamma),\Phi(\Gamma)^\vee,X_*(T)\big)$
is a root datum for $K$. In particular,
this implies that $\Phi(\Gamma)$ is the root system of $K$ with respect to $T$. Since $\Gamma$ is a basis of simple roots of $\Phi(\Gamma)$, we now see that
$e_\Gamma$ is a regular nilpotent element of the Lie algebra $\k=\Lie(K)$. 
Proposition~\ref{regular} shows that the set $\hat{\Lambda}_K(e_\Gamma)$ consists of optimal cocharacters of $e_\Gamma\in\g$. On the other hand, it is well known that
$\hat{\Lambda}_K(e_\Gamma)$ contains a unique element $\lambda=\sum_{\gamma\in\Gamma}a_\gamma\gamma^\vee\in X_*(T)$ with $a_\gamma\in\Z$ which satisfies the conditions $\langle\lambda,\gamma\rangle=2$ for all $\gamma\in\Gamma$. Moreover, 
the coefficients $a_\gamma$ are independent of $p$ and
$\lambda$ lies in $\hat{\Lambda}_{G_\C}(e_{\Gamma,\C})$ when regarded as an element of $X_*(G_\C)$. 
In view of \cite[2.4]{P03} this yields that the nilpotent orbits $\OO(e_\Gamma)\subset\g$ and $\OO(e_{\Gamma,\C})\subset\g_\C$ have the same labels and dimensions.
 \end{proof}
\subsection{Comparing rigid orbits in $\g$ and $\g_\C$}
From now on we assume that $G$ is a simply
connected algebraic group of type ${\rm G}_2$, ${\rm F}_4$, ${\rm E}_6$, ${\rm E}_7$ or ${\rm E}_8$ and denote by $G_\C$ the complex counterpart of $G$. We identify the root system $\Phi$ of $G$ with that of $G_\C$ and 
let $\Pi$ be a basis of simple roots.
Since $p$ is good for $G$, the nilpotent orbits in $\g$ and $\g_\C=\Lie(G_\C)$ are parametrised by their weighted Dynkin diagrams $D\colon\,\Pi\to \{0,1,2\}$. We let $\OO_\C(D)$ be the nilpotent orbit in $\g_\C$ corresponding to $D$. It follows from \cite[Theorem~2.3]{P03}, for example, that
$\dim_k\OO(D)=\dim_\C\OO_\C(D)$. We shall denote this number by $N(D)$.

Combining \cite{DeG13} and \cite{dG-E} with the results of the previous subsection we obtain the following:
\begin{lemma}\label{rigid1} Let $D\colon\,\Pi\to \{0,1,2\}$ be a weighted Dynkin diagram. If the orbit $\OO_\C(D)$ is rigid in $\g_\C$ then the orbit $\OO(D)$ is rigid in $\g$.
\end{lemma}
\begin{proof}
Suppose $\OO_\C(D)$ is rigid in $\g_\C$ and pick $e'\in\OO_\C(D)$ and $e\in\OO(D)$. It follows from \cite{DeG13} that
either the centraliser of $e'$ in $\g_\C$ is perfect or $\OO_\C(D)$ has one of the following types: 

\begin{itemize}
\item [(a)\ ] $\widetilde{\rm A}_1$ in type ${\rm G}_2$;
\smallskip
\item [(b)\ ]$\widetilde{\rm A}_2+{\rm A}_1$ in type ${\rm F}_4$;
\smallskip
\item [(c)\ ] $({\rm A}_3+{\rm A}_1)'$ in type ${\rm E}_7$;
\smallskip
\item [(d)\ ]
${\rm A}_3+{\rm A}_1$, ${\rm A}_5+{\rm A}_1$ or ${\rm D}_5({\rm a}_1)+{\rm A}_2$ in type ${\rm E}_8$.
\end{itemize} 
If the centraliser of $e'$ in $\g_\C$ is perfect then
it follows from Table~\ref{t1} that the same holds for
the centraliser of $e$ in $\g$. In that case the orbit $\OO(D)$ is rigid in $\g$ by Corollary~\ref{almostrigid}.

If $\OO(D)$ is listed in parts~(a), (c), (d) then analysing the tables in \cite[pp.~401--407]{Car93} one finds that
either $\OO(D)$ is the only orbit of dimension $N(D)$ in $\g$ or $G$ is of type ${\rm E}_8$ and $\OO(D)$ is one of ${\rm A}_5+{\rm A}_1$ or ${\rm D}_5({\rm a}_1)+{\rm A}_2$.
If $\OO(D)$ is the only orbit of dimension $N(D)$ in $\g$ then it cannot be induced from a proper Levi subalgebra, because otherwise the same would be true for $\OO_\C(D)$ by Theorem~\ref{thm2.5}. In type ${\rm E}_8$, 
there are two nilpotent orbits of dimension $202$ and they have types
${\rm A}_5+{\rm A}_1$ and ${\rm D}_5({\rm a}_1)+{\rm A}_2$. Moreover, both are rigid in characteristic $0$ by \cite{dG-E}. Therefore, their counterparts in $\N(\g)$ cannot be induced from proper Levi subalgebras by Theorem~\ref{thm2.5}.

Finally, let $\OO(D)$ be as in part~(b). Then $N(D)=36$ and $c(\g_e)=1$ by 
Table~\ref{t4}. Suppose for a contradiction that the orbit $\OO(D)$ is induced from a nilpotent orbit $\OO_0$ in a proper Levi subalgebra $\l=\Lie(L)$ of $\g$. Since in the present case $G$ has type ${\rm F}_4$ and $p$ is good for $G$, the Levi subgroup $L$ cannot have components of type ${\rm A}_{rp-1}$. Applying Proposition~\ref{deform} we then get $c([\l,\l]_{e_0})+\dim \z(\l)\le 1$ which implies that $\dim \z(\l)=1$ and $c([\l_{e_0},\l_{e_0}])=0$ for any $e_0\in\OO_0$. Since $L$ has no components of exceptional types, the orbit $\OO_0$ must be rigid in $\l$ by \cite[Theorem~3(i)]{PT14}. Combining 
Theorem~\ref{thm2.5} and Proposition~\ref{prop2.4} we conclude that
$\OO(D)$ lies in a sheet of $\g$ of dimension $N(D)+\dim\z(\l)=37$. But Theorem~\ref{thm2.5} also shows that there is a bijection $\mathcal{S}\to \mathcal{S}_\C$ between the sheets of $\g$ and $\g_\C$ such that $\dim_k \mathcal{S}=\dim_\C \mathcal{S_\C}$ for all sheets $\mathcal{S}$ of $\g$. By
\cite[Table~10]{dG-E}, there exists only one sheet of dimension $37$ in $\g_\C$ and the nilpotent orbit contained in it has type ${\rm B}_2$. 
Using Proposition~\ref{prop3.2}(ii) it is not hard to observe that the orbit of that type in $\g$ is induced
from the minimal nilpotent orbit in a Levi subalgebra of type ${\rm C}_3$.  In view of Remark~\ref{r5} this implies that $\OO(D)$ cannot lie in the unique sheet of dimension $37$ in $\g$. By contradiction the result follows.
\end{proof}
We now in a position to prove one of the main results of this paper:
\begin{theorem}\label{rigid}
The orbit $\OO(D)$ is rigid in $\g$ if and only if the orbit $\OO_\C(D)$ is rigid in $\g_\C$.
\end{theorem}
\begin{proof}
Let $D\colon\,\Pi\to \{0,1,2\}$ be a weighted Dynkin diagram such that the orbit $\OO(D)$ is rigid in $\g$. In view of Lemma~\ref{rigid1} we need to show that the the orbit $\OO_\C(D)$ is rigid in $\g_\C$. So suppose 
$\OO_\C(D)$ is induced from a proper Levi subalgebra of $\g_\C$. Then $\OO(D)\ne \{0\}$ and thanks to Proposition~\ref{prop3.2}(i) we may assume that $D(\Pi)=\{0,1\}$. If $\OO_\C(D)$ is the only orbit of dimension $N(D)$ in $\g_\C$ then Theorem~\ref{thm2.5} implies that the orbit $\OO(D)$ must be induced in $\g$.

Thus we may assume from now that the orbit $\OO_\C(D)$ is odd, induced, and there are at least two nilpotent orbits of dimension $N(D)$ in $\g_\C$. Looking at the tables in \cite[pp.~401,402]{Car93} one finds out that 
this never happens in types ${\rm G}_2$, ${\rm F}_4$ and ${\rm E}_6$. In type ${\rm E}_7$, there exist only two orbits fulfilling these conditions, namely, ${\rm A}_3+2{\rm A}_1$ and ${\rm D}_4({\rm a}_1)+{\rm A}_1$. In type ${\rm E}_8$ we have to examine more closely the following orbits:
${\rm A}_4+{\rm A}_1$, ${\rm A}_4+{\rm A}_2+{\rm A}_1$,
${\rm D}_6({\rm a}_2)$, ${\rm E}_6({\rm a}_3)+{\rm A}_1$,
${\rm A}_6+{\rm A}_1$, ${\rm A}_7$ and ${\rm D}_7({\rm a}_2)$. 

If $G$ is of type ${\rm E}_8$ and $\OO(D)$ has type ${\rm A}_4+{\rm A}_1$ then there are two nilpotent orbits in $\g$ (and $\g_\C$) of dimension $N(D)=188$. One of these orbits is rigid in $\g_\C$ and hence in $\g$ by Lemma~\ref{rigid1} (it has type $2{\rm A}_3$).
On the other hand, combining Theorem~\ref{thm2.5} with  \cite[Table~8]{dG-E} one observes that
$\g$ must have an {\it induced} orbit of dimension $188$. This rules out the case where $\OO(D)$ has that type. Looking at the diagrams $D(\Gamma)$ attached in \cite{dG-E} to the orbits
labelled ${\rm A}_3+2{\rm A}_1$ in type ${\rm E}_7$ and ${\rm A}_4+{\rm A}_2+{\rm A}_1$, ${\rm E}_6({\rm a}_3)+{\rm A}_1$,
${\rm A}_6+{\rm A}_1$ and ${\rm A}_7$ it type ${\rm E}_8$ one finds out that they satisfy the conditions of Corollary~\ref{princ}. In view of Proposition~\ref{regular} \cite{dG-E} this means that these orbits cannot be rigid in $\g$.

Since there was no obvious way to deal with remaining three cases in types ${\rm E}_7$ and ${\rm E}_8$ these have been checked through GAP. As our calculations in GAP in fact prove the stronger result of Theorem \ref{thm:distrib} without the need to use the present result, we leave the details of this calculation until \S\ref{sec:indcalcs}.\end{proof}

\section{Further computations with nilpotent orbits}

In this section we describe routines performed in GAP which extend those of \cite{DeG13} and \cite{dG-E} to positive characteristic, in particular justifying the tables at the end of the document. The basic problem in obtaining results in arbitrary characteristic is to reduce to a finite list the prime characteristics which must be considered separately. For this, the idea is always to produce an integral matrix such that it encodes this list via the prime divisors of its elementary divisors. 

\subsection{Nilpotent orbit representatives}\label{nilorb}
For our calculations we require (i) a list of nilpotent orbit representatives for bad primes;
(ii) a list of orbit representatives for good primes with associated cocharacters.

Those in (i) are provided by the paper \cite{UGA05}.\footnote{It appears there were two errors transcribing the nilpotent orbit list from the Magma code into the article \cite{UGA05}. Specifically the orbit $F_4(a_3)$ in $F_4$  should read $x_{\alpha_2}+x_{\alpha_1+\alpha_2}+x_{\alpha_2+2\alpha_3}+x_{\alpha_1+\alpha_2+2\alpha_3+2\alpha_4}$; the orbit $A_6^{(2)}$ in $E_8$ should read \[x_{\alpha_1+\alpha_3+\alpha_4}+x_{\alpha_3+\alpha_4+\alpha_5}+x_{\alpha_2+\alpha_3+\alpha_4}+x_{\alpha_2+\alpha_3+\alpha_4}+x_{\alpha_2+\alpha_4+\alpha_5}+x_{\alpha_5+\alpha_6}+x_{\alpha_6+\alpha_7}+x_{\alpha_4+\alpha_5+\alpha_6+\alpha_7}.\]
}

(ii) For this we use \cite{LT11} which provides all the data we require.

Since there are a number of errors in the literature related to nilpotent orbits, we have made reasonable attempts to check the validity of the orbit representatives we use. For example we checked that they had the same Jordan blocks as found in \cite{UGA04}, \cite{UGA05} and \cite{Law95} (in the group case). 
It is worth mentioning that for the exceptional groups considered by \cite{UGA04}, the paper
only deals with the elements of the {\it restricted} nullcone $\N_p(\g)=\{x\in\g\,|\,\,x^{[p]}=0\}$. Hence our computations offer some new insights into the Jordan block structure of those nilpotent elements  $x\in \g$ for which $x^{[p]}\ne 0$ (of course, such elements exist only when $p$ is less that or equal to than the Coxeter number of $G$). It  turned out that as in the group case
(investigated in \cite{Law95})  
there are no coincidences in Jordan block decompositions of representatives of different nilpotent orbits in $\g$ except when $p=7$ where the Jordan block decompositions of nilpotent elements of type $B_3$ and $C_3$ in Lie algebras of type ${\rm F}_4$ are the same. As our calculation is likely to be useful in future work, we state it formally.
\begin{theorem}\label{J-blocks} Let $G$ be a connected reductive $k$-group  satisfying the standard hypotheses and $\g=\Lie(G)$. Then the Jordan block structure of a nilpotent element of $\g$ on the adjoint module coincides with that of a unipotent element of $G$ with the same Dynkin label.\end{theorem}
\begin{proof}
Since $G$ satisfies the standard hypotheses, $p$ is a good prime for $G$.
Our discussion in \S\ref{SH} shows that no generality will be lost by assuming that either
$G$ is a simple exceptional algebraic group or $G$ is a classical group of type other than ${\rm A}_{rp-1}$ or $G=GL_{rp}$. In the first case the theorem follows from our computer-aided calculations whilst in the other two cases it was deduced by McNinch \cite{McN02} from earlier results of Fossum on formal group laws; see \cite{Fos89}.
\end{proof}
It would be interesting to find a non-computational and case-free proof of Theorem~\ref{J-blocks}.

Finally, all our results recover those of De Graaf for generic prime characteristic.
Thus we can declare reasonable confidence in the accuracy our results subject to the correctness of the orbit representatives used in the Magma code in \cite{UGA05}. We thank Dan Nakano for the provision of this data.

\subsection{Induced nilpotent representatives}\label{sec:indcalcs}
Here we show that the nilpotent orbit representatives given in the tables of \cite{dG-E} are still valid in good characteristic and use the results of the previous section to prove this. Let $e\in\g_\Z$ be a nilpotent orbit representative as given in the tables of \cite{dG-E}. We proceed by  having GAP compute the adjoint matrix of each representative over $\Q$ according to a Chevalley basis of $\g_\Z$. Then a routine computes the integral divisors of successive powers of the matrix. For any prime dividing the integral divisors we know that the Jordan block structure of $\ad e$ for $e\in\g=\g_\Z\otimes_{\Z}k$ differs from that in characteristic zero. Let this finite list of primes be $S$. If $p\not\in S$ then since the Jordan block structure of $\ad e$ determines it in good characteristic by \cite{Law95}, we are done in this case. For each good prime of $S$ we have GAP compute the Jordan block structure separately. We find in all cases that the Jordan block structure is the same as that of a nilpotent element in $\g_\C=\g_\Z\otimes_{\Z}\C$ with the same label, as listed in \cite{Law95}.\footnote{Our calculation threw up a misprint in \cite{dG-E}. Specifically the first nilpotent representative for the induced nilpotent orbit $E_7(a_4)$ in $E_8$ has $27$ over two different nodes of the Dynkin diagram. W.~de Graaf kindly redid the calculation and found that the right-most $27$ should be a $44$, and the $44$ (just below the wrong $27$) should be a $14$.}

\begin{proof}[Proof of Theorem~\ref{thm:distrib}]
Each nilpotent element of \cite{dG-E} is of the form $e_\Gamma=e_0+e_1$ where $e_0\in\l$ is a rigid nilpotent orbit of $\l$ and $e_1\in\n$ for $\n$ the nilradical of a parabolic subalgebra $\p=\l+\n$ of $\g$. The orbit dimension $\dim \OO_L(e_0)=\dim[\l,e_0]$ is given by taking the rank of the matrix formed from the coefficients of the images of $e$ under a basis of $\ad \l$. Hence after reduction modulo $p$, the orbit dimension can only go down. Since by \ref{thm2.5} the dimension of the orbit $\Ind_\l^\g\,\OO_L(e_0)$ is given by the formula $\dim \g-\dim\g_e=\dim \g-\dim\l_{e_0}$ we know that if $\dim\g_e$ is independent of good characteristic then $e$ is indeed in the nilpotent orbit which intersects densely with $\OO_L(e_0)+\n$ in good characteristic. But our calculation of the Jordan blocks of $\ad e_\Gamma$ (described in \S\ref{nilorb}) shows that this is the case for all $\Gamma$'s listed in \cite{dG-E} and, moreover, that the orbit $\OO(e_\Gamma)$ always has the same label as its counterpart $\OO(e_{\Gamma,\C})\subset\g_\C$. The exceptional case where $p=7$ and $\g$ is of type ${\rm F}_4$  
does not cause us serious problems because $\g$ contains only two sheets of dimension
$44$ both of which have rank $2$ (this is immediate from  \cite[Table~10]{dG-E} and Remark~\ref{r6}). Proposition~\ref{prop3.2}(ii) implies that one of them
contains the orbit of type ${\rm B}_3$. On the other hand, a closer look at \cite[Table~10]{dG-E} reveals that the sheets of dimension $\ne 44$ cannot contain nilpotent orbits of dimension $42$. This yields that the other sheet of dimension $44$ contains the orbit of type ${\rm C}_3$ (as in the characteristic $0$ case).

Next we observe that $e_0=e_{\Gamma_0}$ and $e_1=e_{\Gamma_1}$ for some full subgraphs $\Gamma_0$ and $\Gamma_1$ of $\Gamma$. Looking through the tables in \cite{dG-E} one finds out that if $G$ is a group of type $\rm E$ then in all cases except two $\Gamma_0$ is a disjoint union of Dynkin graphs (possibly empty). Specifically,
the two exceptional cases occur when $G$ has type ${\rm E}_8$ and $e$ has type ${\rm E_7(a_5)}$ or ${\rm D_5(a_1)+A_1}$. In both cases $\l$ contains a {\it unique} nilpotent orbit whose dimension equals that of $\OO_L(e_0)$. A close look at the tables in \cite{dG-E} shows the same holds
for groups of type ${\rm G}_2$ and ${\rm F}_4$ in all cases of interest.
Thanks to Corollary~\ref{princ} this implies that
the $L$-orbit of $e_0=e_{\Gamma_0}$ has the same
label as the $L_\C$-orbit of $e_{\Gamma_0,\C}$. 

Given a nilpotent orbit $\OO\subset \N(\g)$ we define $$d(\OO):={\rm Card}\,\{\mathcal{S}\,|\,\, \mathcal{S} \mbox { is a sheet of } \g \mbox { containing } \OO\}$$
and denote by $d(\OO_\C)$ the number of sheets of $\g_\C$ containing the nilpotent orbit having the same
Dynkin diagram as $\OO$.
The preceding discussion in conjunction with \cite{dG-E} implies that $d(\OO)\ge d(\OO_\C)$ for any orbit $\OO\subset \N(\g)$.
Since every sheet of $\g$ contains a unique nilpotent orbit (Remark~\ref{r5}) the total number of sheets of $\g$ equals
$\sum_{\OO\subset\N(\g)}\,d(\OO)$. Since Theorem~\ref{thm2.5} established a 
a bijection between the sheets of $\g$ and $\g_\C$, we now deduce that $d(\OO)=d(\OO_\C)$ for any orbit $\OO\subset\N(\g)$.
This completes the proof of Theorem~\ref{thm:distrib}.
\end{proof}
\subsection{(Strong) reachability.}\label{s:str}
We will reduce the problem of classifying the (strongly) reachable elements to a finite calculation which can be performed in GAP. For this we will (i) exhibit a bound on the number of characteristics where the result may differ from the situation in characteristic zero established by de Graaf, then (ii) give a method to deal with any specific prime.

For (ii) we reduce to a calculation over the prime field $\F_p$. Observe that the Lie algebras we deal with are defined over $\Z$. That is, there is a Lie algebra $\g_{\Z}$ over the integers with a basis $\mathcal{B}$, and associated structure constants, such that $\g_p:=\g_{\Z}\otimes_{\Z}\F_p$ and $\g=\g_p\otimes_{\F_p}k$ such that $\g$ and $\g_p$ obtain a basis and associated structure constants by taking $B$ together with its structure constants reduced modulo $p$. 

For this, $\mathcal{B}$ can be taken as a Chevalley basis $\{x_\alpha\,|\,\,\alpha\in\Phi\}\cup\{h_\alpha\,|\,\,\alpha\in \Pi\}$ of $\g_\Z$. Then the tables in \cite[\S5]{UGA05} give a complete set of nilpotent representatives over $k$ in terms of linear combinations of elements $\mathcal{B}$, possibly up to the signs of the coefficients. Since the coefficients are all over $\Z$, these nilpotent elements are also elements of $\g_p$ by reduction mod $p$. Now since $\g_e$ is a restricted subalgebra of $\g$, we may write $\g_e=(\g_p)_e\otimes_{\F_p}k$. Then in order to establish reachability of a nilpotent element $e$, it suffices to check whether we have $e$ contained in the derived subalgebra of $(\g_p)_e$.

This calculation can then be performed by GAP.

For (i), we assume that $p$ is a good prime for $\g$, any bad primes being dealt with in case (ii). We work over the integers, i.e. with the admissible $\Z$-form $\g_\Z$ (here we use our assumption that the group $G$ is simply connected). Fix a nilpotent element $e$. First, we ask GAP to calculate a basis $\mathcal{B}_e$ of $(\g_{\Z})_e$ in terms of linear combinations of elements of $\g$. By the theory of nilpotent elements it is possible to take this basis such that reduction of the coefficients modulo $p$ also gives a basis of $(\g_p)_e$.

Next we take the product of each pair of elements of $\mathcal{B}_e$ and form a matrix $M$ of the $\mathcal{B}$-coefficients of the resulting vectors. Then the dimension of $[(\g_p)_e,(\g_p)_e]$ is the rank of the reduction modulo $p$ of the matrix $M$. Thus, if we take $p$ bigger than any of the elementary divisors of $M$ the result will agree with that over $\Z$. Similar remarks apply to the  subalgebra $\F_pe+[(\g_p)_e,(\g_p)_e]$ of $(\g_p)_e$ whose dimension can be determined by calculating the rank of the matrix $M'$ which is formed by adjoining the row of coefficients of $e$ itself to $M$. Thus for $p$ bigger than any of the elementary divisors of $M$ and $M'$, we have that $e$ is reachable over $k$ if and only if it is reachable over a field of characteristic $0$. More precisely, we establish that the only exceptional primes are less than or equal to $7$.

\subsection{Almost perfect centralisers and $c(\g_e)=\dim(\g_e/[\g_e,\g_e])$}
For a given prime, identifying the orbits which have almost perfect centralisers (i.e. $\g_e=[\g_e,\g_e]+ke$) is a straightforward calculation in $\g_p=\g_\Z\otimes_\Z\F_p$ with GAP: We simply ask GAP to list those nilpotent elements $e$ for which \[\dim(\g_p)_e=\dim[(\g_p)_e,(\g_p)_e]+1\] and
\[\dim\big(\F_pe+[\g_e,\g_e]\big)=\dim (\g_p)_e.\] 
This process deals in particular with the bad primes. The other possible exceptional primes were already calculated: if $\dim[(\g_p)_e,(\g_p)_e]$ or $\dim\big(\F_pe+ [(\g_p)_e,(\g_p)_e]\big)$ differs from the analogous dimension over characteristic zero, then we established in \S\ref{s:str} that $p\leq 7$.

Much the same applies to calculating $c(\g_e)$. Again one need only look in characteristics at most $7$, where the calculation becomes finite, hence easily performed in GAP.

\subsection{Panyushev property.}\label{Pan} A nilpotent elements $e\in\OO(D)$ is said to satisfy {\it Panyushev's property} if the nilradical $\g_e(\lambda_D,\ge 1)$ of $\g_e$ is generated by $\g_e(\lambda_D,1)$ as a Lie algebra. Since 
this definition relies on the so-called associated cocharacters which do not exist for all nilpotent elements in bad characteristic, Panyushev's property is particularly interesting under our assumptions on $G$. 

We assume that $G$ is an exceptional group and $p$ is a good prime for $G$.
From \cite{LT11} we take a nilpotent element $e\in\g$ and associated cocharacter $\tau\in X_*(G)$. For $e\in\OO(D)$ this cocharacter coincides with $\lambda_D$ and it is optimal for $e$ in the sense of the Kempf--Rousseau theory by one of the main results of \cite{P03}. Let $\Phi$ be the root system of $G$ associated with a maximal torus containing $\tau(k^\times)$.
We start by working over rationals and ask GAP to compute the root vectors in $\g_\Q(1):=\g_\Q(\tau,1)$. Let $\mathcal{B}$ be a basis of $\g_\Q$ containing a set of root vectors $\{e_\alpha\,|\,\alpha\in\Phi\}$. Now we form a $(\dim\g_\Q(\tau,1)\times \dim \g_\Q)$-matrix $M$ of the $\mathcal{B}$-coefficients of  $[b,e]$ with $b\in \mathcal{B}\cap \g_\Q(1)$. Then $(\g_\Q)_e(1)$  coincides with the kernel of $\ad\,e\colon\,\g_\Q(1)\to \g_\Q$ which for GAP is $\{v\in \Q^{\dim\g_\Q(\tau,1)}\,|\,\, v\cdot M=0\}$. Since $p$ is good for $G$, we have that $(\g_p)_e=(\g_\Z)_e\otimes_\Z\F_p$ and  $(\g_p)_e(1)=(\g_\Z)_e(1)\otimes_\Z\F_p$. In particular, $(\g_p)_e(1)$ is characteristic independent. 

Next we generate a Lie subring $\m$ of $\g_\Z$ from a $\Q$-basis of $(\g_\Q)_e(1)$ contained in $\g_\Z$ and take a basis $\mathcal{B}_\m$ of the $\Q$-span of $\m$ that lies in $\g_\Z$. We then form the matrix of $\mathcal{B}$-coefficients of the elements of $\mathcal{B}_\m$ and take elementary divisors again. If $p$ is bigger than any of these elementary divisors then the rank of this matrix will give the dimension of the Lie subalgebra of $\g_p$ generated by $(\g_p)_e(1)$. Since it turned out that all elementary divisors appearing in these calculations involve bad primes only, we conclude that the Panyushev property is independent of good characteristic.
\begin{remark} There are six rigid nilpotent orbits in exceptional Lie algebras which do not satisfy Panyushev's property in good characteristic. These orbits are $\tilde{\rm A}_1$ in type ${\rm G}_2$, ${\rm {\tilde A}_2+A_1}$ in type 
${\rm F}_4$, $({\rm A_3+A_1)'}$ in type ${\rm E}_7$, and
${\rm A_3+A_1}$, ${\rm A_5+A_1}$, ${\rm D_5(a_1)+A_2}$ in type ${\rm E}_8$. The above routine was applied to these orbits too in order to determine the smallest number $r$ for which
$\sum_{i=1}^r\,\g_e(\tau,i)$ generates the Lie algebra $\g_e(\tau,\ge 1)$. It turned out that $r=2$ in type ${\rm G}_2$, $r=4$ in type $E_8$ when $e$ has type ${\rm A_5+A_1}$ and $r=3$ in the other four cases. 
All elementary divisors that we encountered in the process turned out to be divisible by $2$ and $3$ only.
\end{remark}

\begin{table}[!htb]
\begin{minipage}{.5\linewidth}\begin{center}\begin{tabular}{l l l l l}
$\g$ & $e$ & $p$ & strong & almost\\\hline
$G_2$  & $\tilde A_1^{(3)}$ & $3$\\
 & $\tilde A_1$ & $2,3$ & $3$ & $p\geq 5$ \\
 & $A_1$ & any & $p\geq 5$\\\hline
$F_4$ & $F_4$ & $3$\\
& $C_3$ & $2$\\
& $C_3(a_1)$ & $2$ \\
& $A_1+\tilde A_2$ & $2,3$ & & $p\geq 5$\\
& $B_2$ & & & $p\geq 3$\\
& $A_2+\tilde A_1$ & $p\geq 3$ & $p\geq 5$ \\
& $\tilde A_2$ & $2$ & $2$ & $p\geq 3$\\
& $A_2$ & & & $p\geq 3$\\
& $A_1+\tilde A_1$ & any & $p\geq 3$\\
& $\tilde A_1^{(2)}$ & $2$\\
& $\tilde A_1$ & any & any\\
& $A_1$ & any & $p\geq 3$\\\hline
$E_6$ & $E_6$ & $3$\\
& $A_5$ & $2$\\
& $A_4+A_1$ & $2,3$\\
& $A_3+A_1$ & $2$\\
& $2A_2+A_1$ & any & $p\geq 5$\\
& $A_2+2A_1$ & any\\
& $2A_2$ & $2$\\
& $A_2+A_1$ & any \\
& $A_2$ & & & any\\
& $3A_1$ & any & $p\geq 3$\\
& $2A_1$ & any\\
& $A_1$ & any & any \\\hline 
\end{tabular}\end{center}\end{minipage}%
\begin{minipage}{.5\linewidth}\begin{center}
\begin{tabular}{l l l l l}
$\g$ & $e$ & $p$ & strong & almost\\\hline
$E_7$ & $E_7$ & $3$\\
 & $E_6$ & $3$\\
 & $D_6(a_1)$ & $2$\\
 & $A_6$ & $2$\\
 & $D_5(a_1)+A_1$ & $2$\\
  & $A_5+A_1$ & $3$ & & $p\geq 5$\\
 & $(A_5)'$ & $2$ & & $p\geq 3$\\
 & $A_4+A_2$ & $2,3$& & $p\geq 5$\\
 & $A_4+A_1$ & any\\
 & $D_4+A_1$ & & & $p\geq 3$ \\
 & $A_3+A_2+A_1$ & $3,5$ & & $p\geq 7$\\
 & $A_3+A_2$ & $2$\\
 & $D_4(a_1)+A_1$ & $2$\\
 & $A_3+2A_1$ & $3$& & $p\geq 5$\\
 & $(A_3+A_1)'$ & $2$ & & $p\geq 3$\\
 & $2A_2+A_1$ & any & $p\geq 5$\\
 & $A_2+3A_1$ & $7$ & & $p\neq 2,7$\\
 & $2A_2$ & $2$& & $p\geq 3$\\
 & $A_3$ & & & $p\geq 3$\\
 & $A_2+2A_1$ & any & $p\geq 3$\\
 & $A_2+A_1$ & any\\
 & $4A_1$ & $p\geq 3$ & $p\geq 3$\\
 & $A_2$ & & & any\\
 & $(3A_1)'$ & any & $p\geq 3$\\
 & $(3A_1)''$ & $3$ & & $p\geq 5$\\
 & $2A_1$ & any & $p\geq 3$\\
 & $A_1$ & any & any\\
 &\\
 \hline
 \end{tabular}\end{center}\end{minipage}
 \smallskip
 \caption
 {Strongly reachable and almost reachable orbits for $G_2$, $F_4$, $E_6$ and $E_7$}\label{t1}\end{table}
\begin{table}[!htb]
\begin{minipage}{.5\linewidth}  
\begin{center}
 \begin{tabular}{l l l l l}
 $\g$ & $e$ & $p$ & strong & almost\\\hline
$E_8$ & $E_8$ & $3,5$\\
& $E_8(a_1)$ & $3$\\
& $E_7$ & $3$\\
& $D_7$ & $2$\\
& $E_6+A_1$ & $3$\\
& $D_7(a_1)$ & $2$\\
& $E_{8}(b_{6})$ & $3$\\
& $(A_{7})^{(3)}$ & $3$\\
& $A_{7}$ & $2,3$ & & $p\geq 5$\\
& $D_{7}(a_{2})$ & $2$\\
& $E_{6}$ & $3$\\
& $A_6+A_1$ & $3,5$ & & $p\geq 7$\\
& $D_{6}(a_{1})$ & $2$\\
& $A_6$ & $2$\\
& $D_{6}(a_{2})$ & $2$ & & $3$\\
& $D_{5}(a_{1})+A_{2}$ & $2$ & & $p\geq 5$\\
& $A_{5}+A_{1}$ & $2,3$ & & $p\geq 5$\\ 
& $A_{4}+A_{3}$ & any & $p\geq 7$\\
& $D_4+A_2$ & $2$\\
& $A_{4}+A_{2}+A_{1}$ & $7$ & & $p\neq 2,5,7$\\
& $D_{5}(a_{1})+A_{1}$ & $2$ & & $p\geq 3$\\
& $A_5$ & $2$ & & $p\geq 3$\\
& $A_{4}+A_{2}$ & $2,3$ & & $p\geq 5$\\
 \hline\end{tabular}\end{center}\end{minipage}%
\begin{minipage}{.5\linewidth}
  \begin{center}\begin{tabular}{l l l l}
$e$ & $p$ & strong & almost\\\hline
$A_{4}+2A_{1}$& any\\
 $2A_{3}$& any & $p\geq 3$\\
 $A_{4}+A_{1}$& any\\
 $D_{4}(a_{1})+A_{2}$ & $2$ & & $p\geq 5$\\
 $D_4+A_1$ & & & $p\geq 3$\\
 $A_{3}+A_{2}+A_{1}$& $p\geq 3$ & $p\geq 3$\\
 $A_3+A_2$ & $2$\\
 $D_{4}(a_{1})+A_{1}$& any & $p\geq 3$\\
 $A_{3}+2A_{1}$& any & $p\geq 3$\\
 $2A_{2}+2A_{1}$&any & $p\geq 5$\\
 $D_4(a_1)$ & & & $2$\\
 $A_3+A_1$ & $2$ & & $p\geq 3$\\
 $2A_{2}+A_{1}$&any & $p\geq 5$\\
 $2A_2$ & $2$ & & $p\geq 3$\\
 $A_{2}+3A_{1}$& any & $p\geq 3$\\
 $A_3$ & & & any\\
 $A_{2}+2A_{1}$& any & $p\geq 3$\\
 $A_{2}+A_{1}$& any & any\\
 $4A_{1}$& any & $p\geq 3$\\
 $A_2$ & & & any\\
 $3A_{1}$& any & $p\geq 3$\\
 $2A_{1}$& any & any\\
 $A_{1}$ & any & any\\\hline 
\end{tabular}
\end{center}\end{minipage}
\smallskip
\caption{Strongly reachable and almost reachable orbits for $E_8$} 
\label{t2}
\end{table}
\begin{table}{\footnotesize\begin{minipage}[t]{0.32\linewidth}\vspace{0pt}\begin{tabular}{l l}
Orbit in $E_8$ & $\dim(\g_e/[\g_e,\g_e])$\\&$p=2,\ 3,\ 5, \geq 7$\\\hline
$E_{8}$ & $12,\ 6,\ 3,\ 8$\\
$E_{8}(a_{1})$ & $12,\ 6,\ 9,\ 7$\\
$E_{8}(a_{2})$ & $12,\ 10,\ 8,\ 6$\\
$E_{8}(a_{3})$ & $12,\ 5,\ 7,\ 7$\\
$E_{8}(a_{4})$ & $14,\ 8,\ 6,\ 6$\\
$E_{7}$ & $11,\ 3,\ 4,\ 4$\\
$E_{8}(b_{4})$ & $11,\ 6,\ 5,\ 5$\\
$E_{8}(a_{5})$ & $11,\ 6,\ 5,\ 5$\\
$E_{7}(a_{1})$ & $11,\ 5,\ 5,\ 5$\\
$E_{8}(b_{5})$ & $11,\ 7,\ 7,\ 7$\\
$(D_{7})^{(2)}$ & $11,-,-,-$\\
$D_{7}$ & $10,\ 2,\ 2,\ 2$\\
$E_{8}(a_{6})$ & $11,\ 6,\ 6,\ 6$\\
$E_{7}(a_{2})$ & $6,\ 6,\ 4,\ 4$\\
$E_{6}+A_{1}$ & $6,\ 5,\ 2,\ 2$\\
$(D_{7}(a_{1}))^{(2)}$ & $6,-,-,-$\\
$D_{7}(a_{1})$ & $5,\ 4,\ 4,\ 4$\\
$E_{8}(b_{6})$ & $7,\ 3,\ 5,\ 5$\\
$E_{7}(a_{3})$ & $8,\ 4,\ 4,\ 4$\\
$E_{6}(a_{1})+A_{1}$ & $3,\ 7,\ 3,\ 3$\\
$(A_{7})^{(3)}$ & $-,6,-,-$\\
$A_{7}$ & $10,\ 3,\ 1,\ 1$\\
$D_{7}(a_{2})$ & $11,\ 2,\ 3,\ 3$\\
$E_{6}$ & $5,\ 3,\ 4,\ 4$\\
$D_{6}$ & $10,\ 2,\ 2,\ 2$\\
$(D_{5}+A_{2})^{(2)}$ & $10,-,-,-$\\
$D_{5}+A_{2}$ & $10,\ 3,\ 3,\ 3$\\
$E_{6}(a_{1})$ & $6,\ 5,\ 4,\ 4$\\
$E_{7}(a_{4})$ & $10,\ 4,\ 3,\ 3$\\
$A_{6}+A_{1}$ & $10,\ 3,\ 1,\ 1$\\
$D_{6}(a_{1})$ & $3,\ 3,\ 3,\ 3$\\
$(A_{6})^{(2)}$ & $10,-,-,-$\\
$A_{6}$ & $5,\ 2,\ 2,\ 2$\\
$E_{8}(a_{7})$ & $10,\ 10,\ 10,\ 10$\\
$D_{5}+A_{1}$ & $5,\ 2,\ 2,\ 2$\\
$E_{7}(a_{5})$ & $5,\ 6,\ 6,\ 6$\\
$E_{6}(a_{3})+A_{1}$ & $5,\ 3,\ 3,\ 3$\\
$D_{6}(a_{2})$ & $5,\ 1,\ 3,\ 3$\\
$D_{5}(a_{1})+A_{2}$ & $4,\ 3,\ 1,\ 1$\\
$A_{5}+A_{1}$ & $5,\ 2,\ 1,\ 1$\\
\end{tabular}\end{minipage}
\begin{minipage}[t]{0.32\linewidth}\vspace{0pt}\begin{tabular}{l l}
Orbit in $E_8$ & $\dim(\g_e/[\g_e,\g_e])$\\&$p=2,\ 3,\ 5, \geq 7$\\\hline
$A_{4}+A_{3}$ & $5,\ 2,\ 2,\ 0$\\
$D_{5}$ & $4,\ 3,\ 3,\ 3$\\
$E_{6}(a_{3})$ & $4,\ 3,\ 3,\ 3$\\
$(D_{4}+A_{2})^{(2)}$ & $3,-,-,-$\\
$D_{4}+A_{2}$ & $9,\ 2,\ 2,\ 2$\\
$A_{4}+A_{2}+A_{1}$ & $3,\ 1,\ 3,\ 1$\\
$D_{5}(a_{1})+A_{1}$ & $1,\ 1,\ 1,\ 1$\\
$A_{5}$ & $3,\ 1,\ 1,\ 1$\\
$A_{4}+A_{2}$ & $4,\ 1,\ 1,\ 1$\\
$A_{4}+2A_{1}$ & $1,\ 1,\ 1,\ 1$\\
$D_{5}(a_{1})$ & $2,\ 2,\ 2,\ 2$\\
$2A_{3}$ & $9,\ 0,\ 0,\ 0$\\
$A_{4}+A_{1}$ & $1,\ 1,\ 1,\ 1$\\
$D_{4}(a_{1})+A_{2}$ & $4,\ 2,\ 1,\ 1$\\
$D_{4}+A_{1}$ & $9,\ 1,\ 1,\ 1$\\
$(A_{3}+A_{2})^{(2)}$ & $9,-,-,-$\\
$A_{3}+A_{2}+A_{1}$ & $9,\ 0,\ 0,\ 0$\\
$A_{4}$ & $2,\ 2,\ 2,\ 2$\\
$A_{3}+A_{2}$ & $4,\ 2,\ 2,\ 2$\\
$D_{4}(a_{1})+A_{1}$ & $2,\ 0,\ 0,\ 0$\\
$A_{3}+2A_{1}$ & $4,\ 0,\ 0,\ 0$\\
$2A_{2}+2A_{1}$ & $4,\ 4,\ 0,\ 0$\\
$D_{4}$ & $3,\ 2,\ 2,\ 2$\\
$D_{4}(a_{1})$ & $1,\ 3,\ 3,\ 3$\\
$A_{3}+A_{1}$ & $2,\ 1,\ 1,\ 1$\\
$2A_{2}+A_{1}$ & $2,\ 2,\ 0,\ 0$\\
$2A_{2}$ & $1,\ 1,\ 1,\ 1$\\
$A_{2}+3A_{1}$ & $2,\ 0,\ 0,\ 0$\\
$A_{3}$ & $1,\ 1,\ 1,\ 1$\\
$A_{2}+2A_{1}$ & $1,\ 0,\ 0,\ 0$\\
$A_{2}+A_{1}$ & $0,\ 0,\ 0,\ 0$\\
$4A_{1}$ & $8,\ 0,\ 0,\ 0$\\
$A_{2}$ & $1,\ 1,\ 1,\ 1$\\
$3A_{1}$ & $2,\ 0,\ 0,\ 0$\\
$2A_{1}$ & $0,\ 0,\ 0,\ 0$\\
$A_{1}$ & $0,\ 0,\ 0,\ 0$\\\end{tabular}\end{minipage}
\begin{minipage}[t]{0.32\linewidth}\vspace{0pt}\begin{tabular}{l l}
Orbit in $E_7$ & $\dim(\g_e/[\g_e,\g_e])$\\&$p=2,\ 3,\ \geq 5$\\\hline
$E_{7}$ & $10,\ 4,\ 7$\\
$E_{7}(a_{1})$ & $10,\ 8,\ 6$\\
$E_{7}(a_{2})$ & $11,\ 7,\ 5$\\
$E_{7}(a_{3})$ & $13,\ 6,\ 6$\\
$E_{6}$ & $11,\ 3,\ 4$\\
$E_{6}(a_{1})$ & $13,\ 6,\ 5$\\
$D_{6}$ & $9,\ 3,\ 3$\\
$E_{7}(a_{4})$ & $9,\ 5,\ 4$\\
$D_{6}(a_{1})$ & $5,\ 4,\ 4$\\
$D_{5}+A_{1}$ & $10,\ 3,\ 3$\\
$(A_{6})^{(2)}$ & $9,-,-$\\
$A_{6}$ & $10,\ 2,\ 2$\\
$E_{7}(a_{5})$ & $10,\ 6,\ 6$\\
$D_{5}$ & $10,\ 3,\ 3$\\
$E_{6}(a_{3})$ & $10,\ 3,\ 3$\\
$D_{6}(a_{2})$ & $5,\ 3,\ 3$\\
$D_{5}(a_{1})+A_{1}$ & $2,\ 2,\ 2$\\
$A_{5}+A_{1}$ & $5,\ 3,\ 1$\\
$(A_{5})'$ & $5,\ 1,\ 1$\\
$A_{4}+A_{2}$ & $6,\ 4,\ 1$\\
$D_{5}(a_{1})$ & $5,\ 3,\ 3$\\
$A_{4}+A_{1}$ & $2,\ 2,\ 2$\\
$D_{4}+A_{1}$ & $8,\ 1,\ 1$\\
$(A_{5})''$ & $4,\ 3,\ 3$\\
$A_{3}+A_{2}+A_{1}$ & $8,\ 1,\ 1$\\
$A_{4}$ & $5,\ 3,\ 3$\\
$(A_{3}+A_{2})^{(2)}$ & $8,-,-$\\
$A_{3}+A_{2}$ & $9,\ 2,\ 2$\\
$D_{4}(a_{1})+A_{1}$ & $4,\ 2,\ 2$\\
$D_{4}$ & $9,\ 2,\ 2$\\
$A_{3}+2A_{1}$ & $4,\ 1,\ 1$\\
$D_{4}(a_{1})$ & $3,\ 3,\ 3$\\
$(A_{3}+A_{1})'$ & $4,\ 1,\ 1$\\
$2A_{2}+A_{1}$ & $4,\ 2,\ 0$\\
$(A_{3}+A_{1})''$ & $3,\ 2,\ 2$\\
$A_{2}+3A_{1}$ & $2,\ 1,\ 1$\\
$2A_{2}$ & $3,\ 1,\ 1$\\
$A_{3}$ & $3,\ 1,\ 1$\\
$A_{2}+2A_{1}$ & $3,\ 0,\ 0$\\
$A_{2}+A_{1}$ & $1,\ 1,\ 1$\\
$4A_{1}$ & $7,\ 0,\ 0$\\
$A_{2}$ & $1,\ 1,\ 1$\\
$(3A_{1})'$ & $8,\ 0,\ 0$\\
$(3A_{1})''$ & $2,\ 1,\ 1$\\
$2A_{1}$ & $2,\ 0,\ 0$\\
$A_{1}$ & $0,\ 0,\ 0$
\\\end{tabular}\end{minipage}}
\smallskip
\caption{Codimension of $[\g_e,\g_e]$ in $\g_e$ for $E_7$ and $E_8$}\label{t3}\end{table}

\begin{table}{\footnotesize\begin{minipage}[t]{0.32\linewidth}\vspace{0pt}\begin{tabular}{l l}
Orbit in $E_6$ & $\dim(\g_e/[\g_e,\g_e])$\\&$p=2,\ 3,\ \geq 5$\\\hline
$E_{6}$ & $5,\ 4,\ 6$\\
$E_{6}(a_{1})$ & $7,\ 8,\ 5$\\
$D_{5}$ & $4,\ 8,\ 4$\\
$E_{6}(a_{3})$ & $6,\ 6,\ 5$\\
$D_{5}(a_{1})$ & $3,\ 5,\ 3$\\
$A_{5}$ & $4,\ 3,\ 2$\\
$A_{4}+A_{1}$ & $5,\ 4,\ 2$\\
$D_{4}$ & $3,\ 2,\ 2$\\
$A_{4}$ & $3,\ 3,\ 3$\\
$D_{4}(a_{1})$ & $2,\ 5,\ 5$\\
$A_{3}+A_{1}$ & $3,\ 4,\ 2$\\
$2A_{2}+A_{1}$ & $3,\ 3,\ 0$\\
$A_{3}$ & $2,\ 2,\ 2$\\
$A_{2}+2A_{1}$ & $1,\ 5,\ 1$\\
$2A_{2}$ & $2,\ 3,\ 2$\\
$A_{2}+A_{1}$ & $1,\ 3,\ 1$\\
$A_{2}$ & $1,\ 1,\ 1$\\
$3A_{1}$ & $2,\ 0,\ 0$\\
$2A_{1}$ & $1,\ 1,\ 1$\\
$A_{1}$ & $0,\ 0,\ 0$\\\end{tabular}\end{minipage}
\begin{minipage}[t]{0.32\linewidth}\vspace{0pt}\begin{tabular}{l l}
Orbit in $F_4$ & $\dim(\g_e/[\g_e,\g_e])$\\&$p=2,\ 3,\ \geq 5$\\\hline
$F_{4}$ & $5,\ 3,\ 4$\\
$F_{4}(a_{1})$ & $5,\ 5,\ 4$\\
$F_{4}(a_{2})$ & $7,\ 4,\ 3$\\
$(C_{3})^{(2)}$ & $8,-,-$\\
$C_{3}$ & $4,\ 2,\ 2$\\
$B_{3}$ & $5,\ 2,\ 2$\\
$F_{4}(a_{3})$ & $7,\ 6,\ 6$\\
$C_{3}(a_{1})^{(2)}$ & $8,-,-$\\
$C_{3}(a_{1})$ & $4,\ 3,\ 3$\\
$(\tilde{A}_{2}+A_{1})^{(2)}$ & $8,-,-$\\
$\tilde{A}_{2}+A_{1}$ & $4,\ 2,\ 1$\\
$(B_{2})^{(2)}$ & $7,-,-$\\
$B_{2}$ & $6,\ 1,\ 1$\\
$A_{2}+\tilde{A}_{1}$ & $7,\ 2,\ 0$\\
$\tilde{A}_{2}$ & $0,\ 1,\ 1$\\
$(A_{2})^{(2)}$ & $7,-,-$\\
$A_{2}$ & $7,\ 1,\ 1$\\
$A_{1}+\tilde{A}_{1}$ & $4,\ 0,\ 0$\\
$(\tilde{A}_{1})^{(2)}$ & $6,-,-$\\
$\tilde{A}_{1}$ & $0,\ 0,\ 0$\\
$A_{1}$ & $6,\ 0,\ 0$\\
\end{tabular}\end{minipage}
\begin{minipage}[t]{0.32\linewidth}\vspace{0pt}\begin{tabular}{l l}
Orbit in $G_2$ & $\dim(\g_e/[\g_e,\g_e])$\\&$p=2,\ 3,\ \geq 5$\\\hline
$G_{2}$ & $3,\ 3,\ 0$\\
$G_{2}(a_{1})$ & $3,\ 3,\ 1$\\
$(\tilde{A}_{1})^{(3)}$ & $-,\ 3,-$\\
$\tilde{A}_{1}$ & $2,\ 0,\ 3$\\
$A_{1}$ & $2,\ 2,\ 2$\\
\end{tabular}\end{minipage}}
\smallskip
\caption{Codimension of $[\g_e,\g_e]$ in $\g_e$ for $G_2$, $F_4$ and $E_6$}\label{t4}\end{table}

{\footnotesize
\bibliographystyle{alpha}


\end{document}